\documentclass[12pt,reqno]{amsproc}

\usepackage[foot]{amsaddr}

\usepackage[margin=.85in]{geometry}
\usepackage{nicefrac, graphicx, mathrsfs}
\usepackage{amsmath, amssymb, amsthm}
\usepackage{caption}
\usepackage{subcaption}
\usepackage{cancel}

\usepackage{appendix}
\usepackage{physics}
\usepackage{mathtools}
\usepackage{enumitem}
\usepackage{xcolor}

\usepackage[sc]{mathpazo}\renewcommand{\mathbf}{\mathbold}
\renewcommand{\mathcal}{\mathscr}

\usepackage{thmtools, thm-restate}

\usepackage{tikz}
\usetikzlibrary{cd}
\usetikzlibrary{matrix,arrows}

\usepackage{xparse}
\usepackage{xifthen}

\usepackage[all,2cell,cmtip]{xy}
\UseAllTwocells

\usepackage{hyperref}
\hypersetup
{
	colorlinks,	%
	citecolor=red,%
	filecolor=blue,%
	linkcolor=blue,%
	urlcolor=red
}

\usepackage{cleveref}

\newtheoremstyle{scthm} 
{}                    
{}                    
{\itshape}                   
{\parindent}                           
{\scshape}                   
{.}                          
{.5em}                       
{}  

\newtheoremstyle{scdef} 
{}                    
{}                    
{}                   
{\parindent}                           
{\scshape}                   
{.}                          
{.5em}                       
{}  
\theoremstyle{scthm}
\newtheorem{theorem}{Theorem}[section]
\newtheorem{corollary}[theorem]{Corollary}
\newtheorem*{theorem*}{Theorem}
\newtheorem*{lemma*}{Lemma}
\newtheorem{lemma}[theorem]{Lemma}

\newtheorem{proposition}[theorem]{Proposition}
\newtheorem{pheorem}[theorem]{Theorem}

\theoremstyle{scdef}
\newtheorem{definition}[theorem]{Definition}
\newtheorem{assumption}[theorem]{Assumption}
\newtheorem{ex}[theorem]{Example}
\newtheorem{notation}[theorem]{Notation}
\newtheorem{remark}[theorem]{Remark}

\newcommand{\set}[1]{\left\{#1\right\}}
\newcommand{\sett}[2]{\qty{#1 \ : \ #2}}

\newcommand{\R}{\mathbb{R}}

\renewcommand{\S}{\mathbb{S}}

\DeclareMathOperator{\HG}{\text{\bf H}}

\newcommand{\vol}{\mathrm{vol}}

\newcommand{\inv}[1]{{{#1}^{-1}}}

\newcommand{\sublevel}[2]{{\inv{#1}(-\infty, {#2}]} }
\newcommand{\sublevelopen}[2]{{\inv{#1}\qty(-\infty, {#2})} }
\newcommand{\suplevel}[2]{{\inv{#1}[{#2}, \infty)}}
\newcommand{\suplevelopen}[2]{{\inv{#1}\qty({#2}, \infty)}}
\newcommand{\intlevel}[3]{{\inv{#1}\qty[{#2}, {#3}]}}

\newcommand{\level}[2]{{\inv{#1}\qty(#2)}}
\newcommand{\ratiograd}[2]{\frac{\norm{\nabla{#1}}}{\norm{\nabla{#2}}}}
\newcommand{\interior}[1]{\mathrm{int}\qty({#1})}
\newcommand{\closure}[1]{\mathrm{cl}\qty({#1})}

\newcommand{\X}{\mathbb{X}}

\renewcommand{\angle}{\measuredangle}
\renewcommand{\phi}{\varphi}

\newcommand{\Con}{\text{\bf Con}}
\newcommand{\GM}[1]{\mathcal{#1}}

\newcommand{\crit}[1]{{\text{\bf Crit}\qty(#1)}}
\newcommand{\ccrit}[1]{{\text{\bf Crit}_c\qty(#1)}}

\newcommand{\med}[1]{{\text{\bf Med}\qty(#1)}}
\newcommand{\Nneighs}[2]{\text{\bf NN}_{#1}\qty({#2})}
\newcommand{\reach}{{\tau}}
\newcommand{\fakereach}{\reach^+}

\newcommand{\fakelfs}[2]{\reach^+_{#1}\qty({#2})}

\newcommand{\lfs}[2]{\reach_{#1}\qty({#2})}
\newcommand{\UP}[1]{\text{\bf UP}\qty({#1})}

\newcommand{\projmap}[2]{\xi_{{#1}}\qty(#2)}

\newcommand{\ball}[2]{B_{#1}\qty({#2})}

\newcommand{\ballc}[2]{B_{#1}\qty[{#2}]}

\newcommand{\twoFF}{\textup{\uppercase\expandafter{\romannumeral 2}}}
\newcommand{\twoff}[2]{\textup{\uppercase\expandafter{\romannumeral 2}}\qty({#1}, {#2})}

\newcommand{\normvec}[1]{\frac{{#1}}{\norm{{#1}}}}

\newcommand{\tp}[2]{T_{#1}{#2}}

\newcommand{\anglegeo}[2]{\angle \qty({#1}, {#2})}
\newcommand{\vel}{{\dot{\gamma}}}
\newcommand{\acc}{\ddot{\gamma}}
\newcommand{\innerprod}[2]{\expval{{#1}, {#2}}}
\newcommand{\Tb}[1]{\mathcal{T}\qty({#1})}

\newcommand{\lezero}{(-\infty, 0]}
\newcommand{\vball}[1]{\mathrm{V}^{(m)}_{\mathrm{B}}\qty({#1})}
\newcommand{\vhs}[2]{\mathbf{V}^{(m)}_{\mathrm{HS}}\qty({#1}, {#2})}
\newcommand{\vhstwo}[3]{\mathbf{V}^{(m)}_{\mathrm{HS2}}\qty({#1}, {#2}, {#3})}

\title{Topological Inference of the Conley Index}

\author{Ka Man Yim}\address{Mathematical Institute, University of Oxford}
\email{yim@maths.ox.ac.uk}
\author{Vidit Nanda}
\email{nanda@maths.ox.ac.uk}

\begin{document}
	
	\maketitle
	\begin{abstract}
		The Conley index of an isolated invariant set is a fundamental object in the study of dynamical systems. Here we consider smooth functions on closed submanifolds of Euclidean space and describe a framework for inferring the Conley index of any compact, connected isolated critical set of such a function with high confidence from a sufficiently large finite point sample. The main construction of this paper is a specific index pair which is local to the critical set in question. We establish that these index pairs have positive reach and hence admit a sampling theory for robust homology inference. This allows us to estimate the Conley index, and as a direct consequence, we are also able to estimate the Morse index of any critical point of a Morse function using finitely many local evaluations.
	\end{abstract}
	
	\section*{Introduction}
	
	There is a pair of spaces at the heart of every topological quest to study gradient-like dynamics.
	Such space-pairs appear, for instance, whenever one encounters a compact $m$-dimensional Riemannian manifold $M$ endowed with a Morse function $f:M \to \R$. The fundamental result of Morse theory \cite{milnor_morse_1973} asserts that if $f$ admits a single critical value in an interval $[a,b] \subset \R$, and if this value corresponds to a unique critical point $p \in M$ of Morse index $\mu$, then the sublevel set
	\[
	M_{f \leq b} := \set{x \in M \mid f(x) \leq b}
	\]
	is obtained from $M_{f\leq a}$ by gluing a closed $m$-dimensional disk $\mathbb{D}$ along a $(\mu-1)$-dimensional boundary sphere $\mathbb{S}$ . The relevant pair\footnote{Several other pairs of spaces would satisfy the same homological properties, e.g., $(M_{f \leq b}, M_{f \leq a})$ itself, or $(M_{a \leq f \leq b}, M_{f=b})$. The goal is usually to find the simplest and most explicit space-pair.} in this case is $(\mathbb{D},\mathbb{S})$, and it follows by excision that there are isomorphisms
	\[
	\HG_\bullet\left(M_{f \leq b},M_{f \leq a}\right) \simeq  \HG_\bullet\left(\mathbb{D},\mathbb{S}\right)
	\] of (integral) relative homology groups. Thus, when working over field coefficients, the homology groups of $M_{f \leq b}$ are obtained by altering those of $M_{f \leq a}$ in precisely one of two ways --- either the $\mu$-th Betti number is incremented by one, or the $(\mu-1)$-st Betti number is decremented by one.
	
	Attempts to extend this story beyond the class of Morse functions run head-first into two significant complications --- first, since the critical points need not be isolated, one must confront arbitrary critical subsets; and second, there is no single number analogous to the Morse index which completely characterises the change in topology from $M_{f \leq a}$ to $M_{f \leq b}$. The first complication is encountered in Morse-Bott theory \cite{morsebott}, where the class of admissible functions is constrained to ones whose critical sets are normally nondegenerate submanifolds of $M$. The second complication is ubiquitous in Goresky-MacPherson's stratified Morse theory \cite{stratmorse}, where the class of admissible functions is constrained to those which only admit isolated critical points. In both cases, there are satisfactory analogues of $(\mathbb{D},\mathbb{S})$ obtained by separately considering tangential and normal Morse data; however, the constraints imposed on functions $M \to \mathbb{R}$ in these extensions of Morse theory are far too severe from the perspective of dynamical systems.

	\subsection*{The Conley Index} Conley index theory \cite{conley_isolated_1978, mischaikow_conley_2002}  provides a powerful generalisation of Morse theory which has been adapted to topological investigations of dynamics.
	
	Consider an arbitrary smooth function $f:M \to \R$ and the concomitant gradient flow $\sigma:\R \times M  \to M$. A subset $S \subset M$ is {\em invariant} under $f$ if $\sigma(t,x)$ lies in $S$ whenever $(t,x)$ lies in $\R \times S$; and such an $S$ is {\em isolated} if there exists a compact subset $N \subset M$ containing $S$ in its interior, so that $S$ is the largest invariant subset of $f$ inside $K$. Assuming that $S$ is isolated in this sense, let $N_- \subset N$ be any compact subset disjoint from the closure of $S$ so that any flow line which departs $N$ does so via $N_-$. Pairs of the form $(N,N_-)$ are called {\em index pairs} for $S$, and the relative homology $\HG_\bullet(N,N_-)$ does not depend on the choice of index pair; this relative homology is called the {\bf Conley index} of $S$. The notion of index pairs subsumes not only the pair $(\mathbb{D},\mathbb{S})$ from Morse theory, but also the analogous local Morse data for Morse-Bott functions and stratified Morse functions.
	
	The Conley index enjoys three remarkable properties as an algebraic-topological measurement of isolated invariant sets: \begin{enumerate}
		\item being relative homology classes, Conley indices are efficiently computable \cite{kmm, hmmn};
		\item if $\HG_\bullet(N,N_-)$ is nontrivial for an index pair, then one is guaranteed the existence of a nonempty invariant set in the interior of $N$; and finally,
		\item the Conley index of an isolated invariant set $S$ remains constant across sufficiently small perturbations of $f:M \to \R$ even though $S$ itself might fluctuate wildly.
	\end{enumerate}
	As a result of these  attributes, the Conley index has found widespread applications to several interesting dynamical problems across pure and applied mathematics. We have no hope of providing an exhaustive list of these success stories here, but we can at least point the interested reader to the study of fixed points of Hamiltonian diffeomorphisms \cite{conzehn}, travelling waves in predator-prey systems \cite{predprey}, heteroclinic orbits in fast-slow systems \cite{fastslow}, chaos in the Lorenz equations \cite{mmchaos}, and local Lefschetz trace formulas for weakly hyperbolic maps \cite{gmlefs}.
	
	\subsection*{Topological Inference} An enduring theme within applied algebraic topology involves recovering the homology of an unknown subset $X$ of Euclidean space $\R^d$ with high confidence from a finite point cloud $P \subset \R^d$ that lies on, or more realistically, near $X$. This task is impossible unless one assumes some form of regularity on $X$  --- no amount of finite sampling will unveil the homology groups of Cantor sets and other fractals. 
	
	The authors of \cite{niyogi_finding_2008} consider the case where $X$ is a compact Riemannian manifold and $P$ is drawn uniformly and independently from either $X$ or from a small tubular neighbourhood of $X$ in $\mathbb{R}^d$. Their main result furnishes, for sufficiently small radii $\epsilon > 0$ and probabilities $\delta \in (0,1)$, an explicit lower bound $B = B_X(\epsilon,\delta)$. If the cardinality of $P$ exceeds $B$, then it holds with probability exceeding $(1-\delta)$ that the homology of $X$ is isomorphic to that of the union $P^\epsilon$ of $\epsilon$-balls around points of $P$. Similar results have subsequently appeared for inferring homology of manifolds with boundary \cite{wang_topological_2020}, of a large class of Euclidean compacta \cite{chazal_sampling_2009}, and of induced maps on homology \cite{fmn}.
	
	A crucial regularity assumption underlying all of these results is that the map induced on homology by the inclusion $X \hookrightarrow X^{\epsilon}$ is an isomorphism for all suitably small $\epsilon > 0$. When $X$ is smooth, this can be arranged by requiring the radius $\epsilon$ to be controlled by the injectivity radius of the embedding $X \hookrightarrow \R^d$, often called the {\em reach} of $X$ --- see \cite{federer_curvature_1959}.

	\subsection*{This Paper} Here we consider a compact, connected and isolated critical set $S$ of a smooth function $f:M \to \R$ defined on a closed submanifold $M \subset \R^d$. Our contributions are threefold:
	\begin{enumerate}
		\item we construct a specific index pair $(\GM{N},\GM{N}_-)$ for $S$ in terms of auxiliary data pertaining to the isolating neighbourhood of $S$ in $M$; moreover,
		\item we establish that both $\GM{N}$ and $\GM{N}_-$ have positive reach when viewed as subsets of $\R^d$; and finally,
		\item we provide a sampling theorem for inferring the Conley index $\HG_\bullet(\GM{N},\GM{N}_-)$ from finite point samples of $\GM{N}$ and $\GM{N}_-$.
	\end{enumerate} 
	
	The auxiliary data required in our construction of $(\GM{N},\GM{N}_-)$ is a smooth real-valued function $g$, which is defined on an isolating neighbourhood and whose vanishing locus equals $S$. These are not difficult to find --- one perfectly acceptable choice of $g$ is the norm-squared of the gradient $\norm{\nabla f}^2$. Using any such $g$ along with a smoothed step function, we construct a perturbation $h:M \to \R$ of $f$ which agrees with $f$ outside the isolating neighbourhood. The set $\GM{N}$ is then obtained by intersecting a sublevel set of $f$ with a superlevel set of $h$; and similarly, $\GM{N}_-$ is obtained by intersecting the same sublevel set of $f$ with an interlevel set of $h$. The endpoints of all intervals considered in these (sub, super and inter) level sets are regular values of $f$ and $h$, i.e., $\nabla f$ and $\nabla h$. Here is a simplified version of our main result, summarising \Cref{prop:indexpairsample} and \Cref{thm:ebound}.
	
	\begin{theorem*}[A]
		Let $(\GM{N},\GM{N}_-)$ be our constructed index pair for $S$. Assume that $\X \subset \GM{N}$ is a (uniform, independent) finite point sample, and set $\X_- := \X \cap \GM{N}_-$. Then:
		\begin{enumerate}
			\item If the density of $(\X, \X_-)$ in $(\GM{N},\GM{N}_-)$ exceeds an explicit threshold $t_1$, then $\HG_\bullet(\X^\epsilon, \X^\epsilon_-)$ is isomorphic to the Conley index of $S$ over an open interval of choices of $\epsilon$;
			\item A point sample with sufficient density can be realised with high probability $1-\kappa$ from a uniform, i.i.d. sample of $(\GM{N}, \GM{N}_-)$, if the number of points exceeds a threshold $t_2$; and 
			\item The thresholds $t_1$ and $t_2$ depend only on the reach of the manifold, $C^1$ data of $f$ and $g$ on the isolating neighbourhood, and bounds on the norm of the second derivatives of $f$ and $g$. 
		\end{enumerate}
	\end{theorem*}

	An essential step in our proof involves showing that $\GM{N}$ and $\GM{N}_-$ have positive reach. Our strategy for establishing this fact is to prove a considerably more general result, which we hope will be of independent interest. We call $E \subset \R^d$ a {\bf regular intersection} if it can be written as
	\[
	E = \bigcap_{i=1}^\ell \level{f}{I_i}
	\]
	for some integer $\ell > 0$; here each $f_i:M \to \R$ is a smooth function with $0$ a regular value\footnote{That is, the gradient $\nabla_x f_i$ is nonzero at $x = 0$}, and each $I_i$ is either the point $\set{0}$ or the interval $(-\infty,0]$. The geometry of such intersections is coarsely governed by two positive real numbers $\mu$ and $\Lambda$ --- here $\mu$ bounds from below the singular values of all Jacobian minors of $(f_1,\ldots,f_\ell):M \to \R^\ell$, while $\Lambda$ bounds from above the operator norm of all the Hessians $Hf_i$. 
	
	\begin{lemma*}[B]
		Every $(\mu,\Lambda)$-regular intersection of $\ell$ smooth functions $M \to \R$ has reach $\reach$ bounded from below by
		\[ \frac{1}{\reach} \leq \frac{1}{\reach_M}  + \sqrt{\ell}\cdot\frac{\Lambda}{\mu},\]
		where $\reach_M$ is the reach of $M$.
	\end{lemma*}
	\noindent See \Cref{lem:regreachbound} for the full statement and proof of this result. 
	
	\subsection*{Related Work} Our construction of the index pair $(\GM{N},\GM{N}_-)$ for an isolated critical set $S$ is inspired by Milnor's construction for the case where $S$ is a critical point of a Morse function \cite[Secion I.3]{milnor_morse_1973}. Index pairs for isolated critical points of smooth functions have been thoroughly explored by Gromoll and Meyer  \cite{gromoll_differentiable_1969}; the work of Chang and Ghoussoub \cite{chang_conley_1996} provides a convenient dictionary between Conley's index pairs and a generalised version of these Gromoll-Meyer pairs. Also close in spirit and generality to our $(\GM{N},\GM{N}_-)$ are the {\em systems of Morse neighbourhoods} around arbitrary isolated critical sets in the recent work of Kirwan and Penington \cite{kp}.
	
	\subsection*{Outline} In \Cref{sec:conind} we briefly introduce index pairs and the Conley index. In \Cref{sec:constrindxpr}, we give an explicit construction of $(\GM{N}, \GM{N}_-)$. \Cref{sec:geometry} is devoted to proving Lemma (B). In \Cref{sec:hominfidxpr}, we specialise  the above results for regular intersections to our index pairs $(\GM{N}, \GM{N}_-)$ --- in particular, we derive a sufficient sampling density for the recovery of the Conley index and give a bound on the number of uniform independent point samples required to attain this density with high confidence.
	
	\section{Conley Index Preliminaries} \label{sec:conind}
	
	The definitions and results quoted in this section are sourced from Sections III.4 and III.5 of Conley's monograph \cite{conley_isolated_1978}; see also Mischaikow's survey \cite{mischaikow_conley_2002} for a gentler introduction to this material.\footnote{While the Conley index is defined for isolated invariant sets of any flow $\R \times M \to M$ or map $M \to M$, we have confined our presentation here to gradient flows as those are directly relevant to this paper.}
	
	Let $m \leq d$ be a pair of positive integers, and consider a closed $m$-dimensional Riemannian submanifold of $d$-dimensional Euclidean space $\R^d$. Throughout this paper, we fix a smooth function $f:M \to \R$ and denote by $\nabla_x f$ its gradient evaluated at a point $x \in M$. The {\em gradient flow} of $f$ is the solution $\sigma:\R \times M \to M$ to the initial value problem:
	\begin{equation}
	\pdv{\sigma(t,x)}{t} = -\nabla_x f  \qand \sigma(0,x) = x. \label{eq:gradflow}
	\end{equation}
	We call $x \in M$ a \emph{critical point} of $f$ if $\nabla_x f = 0$, whence $\sigma(\R, x) = x$. Let $\crit{f}$ denote the set of critical points of $f$, and $\ccrit{f}$ the set of compact connected components of $\crit{f}$. More generally, a subset $S \subset M$ is {\em invariant} under $\sigma$ whenever $\sigma(\R,S) \subset S$. We say that $S$ is \emph{isolated} if there exists a compact set $K \subset M$ such that $S$ is the maximal invariant set of $\sigma$ lying in the interior of $K$; explicitly, we must have
	\begin{align}\label{eq:isonbr}
	S = \set{x \in K \mid \sigma(\R,x) \subset K} \subset \interior{K}.
	\end{align} Any such $K$ is called an {\em isolating neighbourhood} of $S$.
	
	\begin{definition}\label{def:indpair} Let $N_- \subset N$ be  pair of compact subsets of $M$. We call $(N,N_-)$ an {\bf index pair} for the isolated invariant set $S$ if the following axioms are satisfied:
		\begin{enumerate}[start=1,label={\bfseries \upshape (IP\arabic*)}, ref={\bfseries \upshape IP(\arabic*)}]
			\item \label{IP1} the closure $\closure{N \setminus N_-}$ is an isolating neighbourhood of $S$;
			\item \label{IP2} the set $N_-$ is \emph{positively invariant} in $N$: that is, for any $(t,x) \in \R_{> 0} \times N_-$ with $\sigma([0,t],x) \subset N$, we have $\sigma([0,t],x) \subset N_-$;
			\item \label{IP3} the set $N_-$ is an \emph{exit set}; namely, if $\sigma(t,x) \notin N$ for $(t,x) \in \R_{>0} \times N$, then there exists some $s \in [0,t]$ with $\sigma([0,s],x) \subset N$ and $\sigma(s,x) \in N_-$.
		\end{enumerate}
	\end{definition}
	
	Every isolated invariant set $S$ of $\sigma$ admits an index pair $(N,N_-)$ --- see \cite[Section III.4]{conley_isolated_1978} for a proof. The content of \cite[Section III.5]{conley_isolated_1978} is that if $(L,L_-)$ is any other index pair for $S$, then the pointed homotopy types of $N/N_1$ and $L/L_1$ coincide. As a result, the relative integral homology groups $\HG_\bullet(N,N_-)$ and $\HG_\bullet(L,L_-)$ are isomorphic and the following notion is well-defined.
	
	\begin{definition}
		The (homological) {\bf Conley index} of an isolated invariant set $S$, denoted $\Con_\bullet(S)$, is the relative homology $\HG_\bullet(N,N_-)$ of any index pair for $S$.
	\end{definition}
	
	It follows immediately from the additivity of homology that if $S$ decomposes as a disjoint union $\coprod_i S_i$ of isolated invariant subsets, then $\Con_\bullet(S)$ is isomorphic to the direct sum $\bigoplus_i\Con_\bullet(S_i)$. Therefore, it suffices to restrict attention to the case where $S$ is connected. In this paper we consider only the special case $S \in \ccrit{f}$, i.e., isolated invariant sets which are compact, connected and critical. It follows that the restriction of $f$ to $S$ is constant, and we will assume henceforth (without loss of generality) that $f(S) = 0$.
	
	\section{Constructing Index Pairs }
	\label{sec:constrindxpr}
	Let $M \subset \R^d$ be a closed Riemannian submanifold and $f:M \to \R$ a smooth function. We consider a (connected, compact) critical set $S \in \ccrit{f}$ with $f(S) = 0$ and isolating neighbourhood $K$, as described in \cref{eq:isonbr}. For each $x$ in $M$, we write $H_xf$ to denote the Hessian matrix of second partial derivatives of $f$ evaluated at $x$. Our goal in this section is to explicitly construct an index pair for $S$ in the sense of Definition \ref{def:indpair}.
	
	\begin{definition}\label{def:g}
		A smooth\footnote{Since $K$ has a boundary, we mean that $g$ is smooth on an open subset of $M$ containing $K$.} map $g:K \to \R$ is called a {\bf bounding function} for $S$ if there exists a pair of real numbers $r_0 < r_1$ for which the following properties hold:
		\begin{enumerate}[label={\bfseries \upshape (G\arabic*)}]
			\item $S \subset \sublevelopen{g}{r_0}$ ; \label{G1}
			\item $\sublevel{g}{r_1}$ is compact; \label{G2}
			\item $S = \sublevel{g}{r_1} \cap \crit{f}$; \label{G3}
			\item $\intlevel{g}{r_0}{r_1} \cap \crit{g} = \emptyset$ \label{G4}.
		\end{enumerate}
		We call $[r_0,r_1]$ a {\em regularity interval} for the bounding function $g$.
	\end{definition}
	
	Bounding functions always exist for isolated sets in $\ccrit{f}$ --- one convenient choice is furnished by the normsquare of the gradient $\nabla f$ of $f$, 
	\begin{lemma} The function $g:K \to \R$ given by $g(x) = \norm{\nabla_xf}^2$ is a bounding function for $S$.
		\begin{proof}
			Writing $\partial K$ for the boundary of $K$ in $M$, set $s = \sup_{x \in \partial K} \norm{\nabla_xf}^2$ and note that $s > 0$ because $K$ is an isolating neighbourhood of $S$. Since critical sets of smooth functions are closed, and $K$ -- by virtue of being an isolating neighbourhood -- is compact, we have that $\crit{g} \cap K$ is compact. As $g$ is continuous, $g(\crit{g})$ is compact in $\R$, and thus the regular values of $g$ are open in $\R$. Applying Sard's theorem, regular values of $g$ then form an open dense subset of $[0,s]$. Consequently, there exists an interval of regular values $[r_0, r_1] \subset (0, s]$ of $g$. Since $r_0 > 0$, and $S$ is the only set of critical points in $K$, \Cref{G1,G2,G3,G4} of Definition \Cref{def:g} are trivially satisfied. 
		\end{proof}
	\end{lemma}
	
	We further assume knowledge of the following numerical data.
	\begin{assumption} For a given bounding function $g:K \to \R$ for $S$, we assume:
		\begin{enumerate}[label={\bfseries \upshape (G\arabic*)}, start = 5]
			\item \label{G5} There is a constant $q_0  >0 $ so that the inequality
			\[
			\ratiograd{f}{g} \frac{(r_1 - r_0)}{2} \geq q_0
			\] holds on $\intlevel{g}{r_0}{r_1}$;
			\item \label{G6} There are regular values $(\alpha, r)$ and $(\alpha, s)$ of $(f,g): K \to \R^2$ satisfying
			\begin{align*}
			0 < \alpha \leq \frac{q_0}{2}  \qand
			\frac{r_0+r_1}{2} < s \leq r < r_1.
			\end{align*}
		\end{enumerate}
	\end{assumption}
	
	\begin{remark} When using $g = \norm{\nabla f}^2$, we have $\nabla g = 2 \cdot Hf \cdot \nabla f$, where $Hf$ is the Hessian of $f$. As long as this Hessian remains nonsingular on $\intlevel{g}{r_0}{r_1}$, the two assumptions above are readily satisfied. In particular, we can rephrase \Cref{G5} to the statement that
		\[
		\frac{\norm{Hf \cdot \nabla f}}{\norm{\nabla f}} \leq \frac{(r_1-r_0)}{q_0};
		\]
		Since the left side is bounded by the operator norm of $Hf$, if we set 
		\[
		b = \inf_{g(x) \in [r_0,r_1]}\|H_xf\|
		\] then any $q_0 \leq \frac{(r_1-r_0)}{b}$ suffices. Similarly, the function $(f,\norm{\nabla f}^2):K \to \R^2$ is singular at $x \in K$ if and only if $\nabla_xf$ is an eigenvector of $H_xf$, so points of $K$ chosen at random will generically be regular.
	\end{remark}
	
	A regularity interval $[r_0,r_1]$ for $g$ may be used to construct a smooth step function which decreases from $q_0$ to $0$. In turn, the function $q$ facilitates the construction of a local perturbation of $f$ near $S$, which we call $h$. This perturbation $h$ is the last piece of information required to construct an index pair for $S$.
	
	\begin{definition}\label{def:qh}
		The {\bf step function} $q:\R \to \R_{\geq 0}$ is defined as
		\begin{align} \label{eq:bump}
		q(t) = \begin{cases}
		q_0 & t < r_0\\
		{q_0}\qty(1 + \exp(\frac{r_1-r_0}{r_1-t} + \frac{r_1-r_0}{r_0-t}))^{-1}& t \in [r_0,r_1] \\
		0 &  t > r_1
		\end{cases},
		\end{align}
		and the {\bf $q$-perturbation} of $f$ is the smooth function $h:M \to \R$ given by
		\begin{align}\label{eq:GMh}
		h(x) = \begin{cases}
		f(x) + q\left(g(x)\right) & \text{if }x \in K \\
		f(x) & \text{otherwise}
		\end{cases}.
		\end{align}
		\Cref{H3} of \Cref{lem:hdef} shows that $h$ is smooth despite its piecewise definition: $f$ only disagress with $h$ on a strict subset of the interior of $K$ .
	\end{definition}
	
	Define the constants
	\begin{align}\label{eq:betagamma}
	\beta = \alpha + q(r) \qand
	\gamma = \alpha + q(s),
	\end{align}
	where $\alpha$ has been chosen in \Cref{G5} while $r$ and $s$ are chosen in \Cref{G6}. Let $(\GM{N},\GM{N}_-)$ be the pair of spaces given by
	\begin{align}
	\GM{N} = \sublevel{f}{\alpha} \cap \suplevel{h}{\beta} \qand \GM{N}_- =  \sublevel{f}{\alpha} \cap \intlevel{h}{\beta}{\gamma}. \label{eq:GMM}
	\end{align}
	(Note that $\GM{N}_- \subset \GM{N} \subset K$ holds by construction). Here is the main result of this section.
	
	\begin{figure}
		\centering
		\begin{subfigure}[t]{0.95\textwidth}
			\centering
			\includegraphics{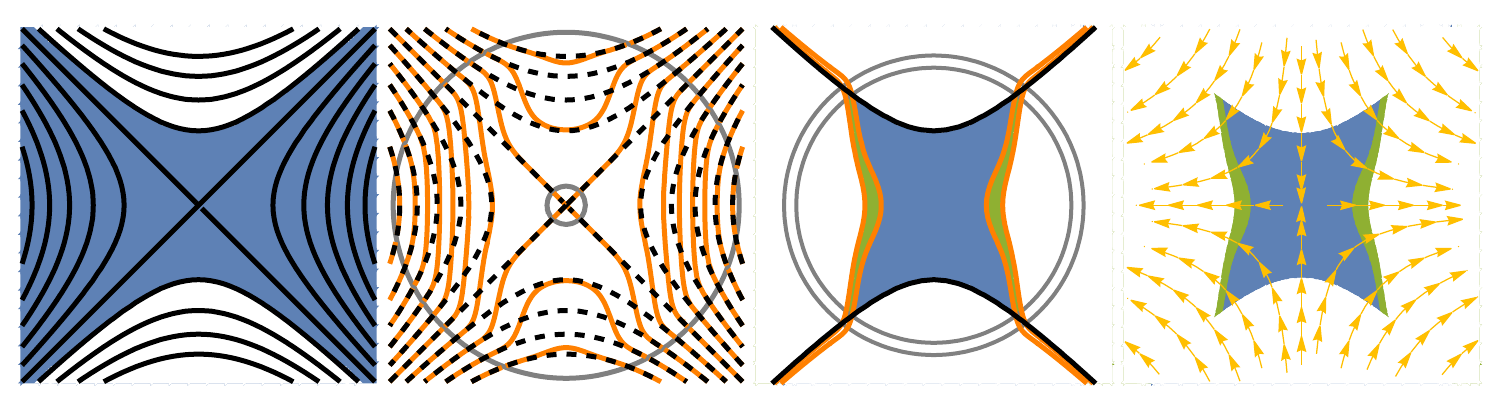}
			\caption{$f = -x^2+ y^2$ with $S = 0$}
		\end{subfigure}
		\begin{subfigure}[t]{0.95\textwidth}
			\centering
			\includegraphics{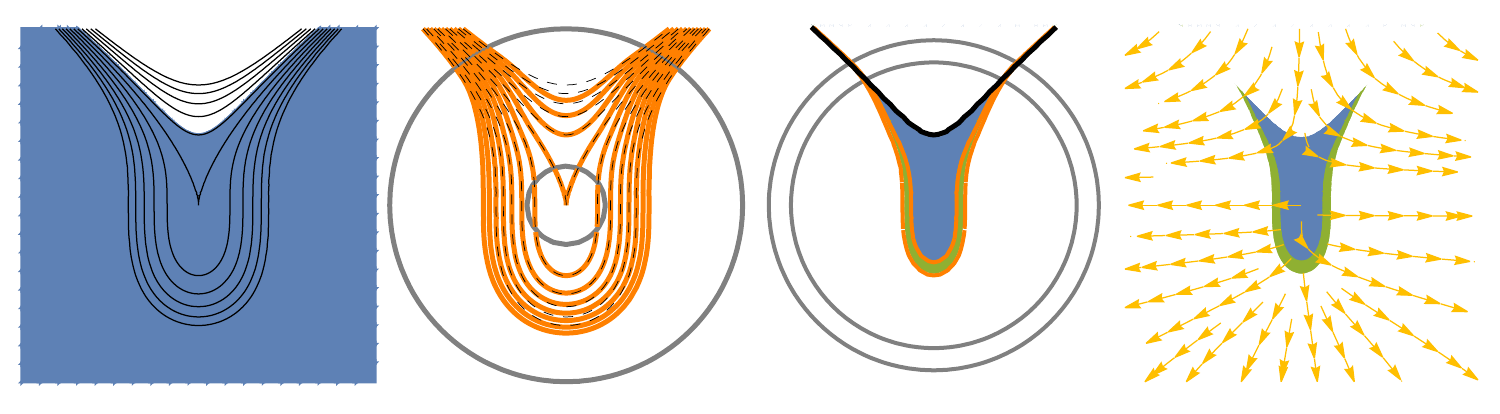}
			\caption{$f = -x^2+ y^3$ with $S = 0$}
		\end{subfigure}
		\caption{Left: the sublevel set  $\sublevel{f}{\alpha}$; Centre left: orange contour lines correspond to those of $h$, black dashed contour lines correspond to those of $f$, and grey contour lines correspond to $g = r_0$ and $g = r_1$. Note how outside $\sublevel{g}{r_1}$, the functions $f$ and $h$ coincide, while within $\sublevel{g}{r_1}$, the local addition of a non-zero $q(g)$ perturbation causes the contours of $h$ to deform and deviate away from $f$. Centre right: an example of $(\GM{N}, \GM{N}_-)$. The green region corresponds to $\GM{N}_-$ and the union of the green and blue regions constitute $\GM{N}$. The orange lines are contour lines of $h = \beta$ and $h = \gamma$; the black lines are the contour lines of $f = \alpha$; and the grey lines are contour lines of $g = r$ and $g = s$. Right: $(\GM{N}, \GM{N}_-)$ depicted with a stream plot of the flow along $-\nabla_x f$ superimposed.}
		\label{fig:indexpair}
	\end{figure}
	
	\begin{theorem} \label{thm:indexpair} The pair $(\GM{N}, \GM{N}_-)$ from \cref{eq:GMM} is an index pair for $S$.
	\end{theorem}
	
	\begin{proof}
		We show that \Cref{IP1} is satisfied in \Cref{lem:IP1}, and \Cref{IP2,IP3} are satisfied in \Cref{lem:IP23} below.
	\end{proof}
	
	Before proceeding to the Lemmas which establish \Cref{thm:indexpair}, we outline some relevant features of the function $h$ from \cref{eq:GMh}.

	\begin{lemma} \label{lem:hdef} The function $h:M \to \R$ satisfies the following properties:
		\begin{enumerate}[start=1,label={\bfseries \upshape (H\arabic*)}]
			\item \label{H1} $f(x) \leq h(x)$;
			\item \label{H2} $g(x) \leq r_0 \iff h(x) = f(x) + q(0)$;
			\item \label{H3} $g(x)  \geq r_1 \iff h(x) = f(x)$; 
			\item \label{H4} $\expval{\nabla {h}, \nabla{f}} \geq 0$ with equality only attained on $\crit{f}$; and,
			\item \label{H5} $\crit{h} = \crit{f}$.
		\end{enumerate}
	\end{lemma}
	\begin{proof}
		The only properties here which don't follow directly from \Cref{def:qh} are \Cref{H4} and \Cref{H5}. For \Cref{H4}, note by \cref{eq:GMh} that $\nabla{h} = \nabla{f}$ holds outside $\intlevel{g}{r_0}{r_1}$ since the derivative of $q$ vanishes in this region. So we consider $x \in \intlevel{g}{r_0}{r_1}$, and calculate
		\[
		\nabla_xf = \nabla_xf + q'(g(x)) \cdot \nabla_xg.
		\]
		It is readily checked that $|q'(t)|$ is maximised at $t=\frac{r_0+r_0}{2}$, where its value is $\frac{2q_0}{r_1-r_0}$. 
		Therefore,
		\begin{align*}
		\expval{\nabla_x{h}, \nabla_x{f}} &= \expval{\nabla_x{f}, \nabla_x{f}} + q'(g(x)) \expval{\nabla_x{g}, \nabla_x{f}}   \\
		& \geq  \norm{\nabla_x{f}} \norm{\nabla_x{g}}\qty(\ratiograd{{}_xf}{{}_xg} - \abs{q'(g(x))}) & \qq{Cauchy-Schwarz, \cref{G4}} \\
		& \geq  \norm{\nabla_x{f}} \norm{\nabla_x{g}}\qty(\ratiograd{{}_xf}{{}_xg} -\frac{2q_0}{r_1-r_0}) & \qq{\text{Bound on }$|q'|$}\\
		&  > 0 & \qq{\cref{G5}}
		\end{align*}
		As a consequence of \cref{H4}, we know that $\nabla{h} \neq 0$ on the set $\intlevel{g}{r_0}{r_1}$. Since $\nabla{h} = \nabla{f}$ whenever $g \leq r_0$ or $g \geq r_1$, have $\crit{f} = \crit{h}$, as required by \Cref{H5}.
	\end{proof}
	
	The next result forms the first step in our proof of \Cref{thm:indexpair}.
	
	\begin{lemma} \label{lem:IP1} The closure  $\closure{\GM{N} \setminus \GM{N}_-}$ is an isolating neighbourhood of $S$.
	\end{lemma}
	\begin{proof} Before verifying \cref{eq:isonbr} with $K := \closure{\GM{N} \setminus \GM{N}_-}$, we first check that the closed set $K$ is compact by confirming that the ambient set $\GM{N}$ is bounded. To this end, note that for any $x \in \GM{N}$, we have:
		\begin{align*}
		h(x) &= f(x) + q(g(x)) & \text{by \cref{eq:GMh}} \\
		&\leq \alpha + q(g(x)) & \text{by  \cref{eq:GMM}}
		\end{align*}
		A second appeal to \cref{eq:GMM} gives $h(x) \geq \beta$ for $x \in \GM{N}$, whence $\beta-\alpha \leq q(g(x))$. But $\beta-\alpha = q(r)$ by \cref{eq:betagamma} and $q$ is strictly decreasing on $[r_0,r_1]$, which forces $g(x) \leq r \leq r_1$. Thus $\GM{N}$ lies within $\sublevel{g}{r_1}$, which is compact by \Cref{G2} and hence bounded.
		
		Next we establish that $S$ lies in the interior of $K$ by showing that $S \subset \sublevelopen{f}{\alpha} \cap \suplevelopen{h}{\gamma}$. For this purpose, note that $f(S) = 0$ by assumption and $\alpha > 0$ by \Cref{G5}, so $S \in \sublevel{f}{\alpha}$ is immediate. And since $S \subset \sublevel{g}{r_0}$ by \Cref{G2}, we have from \Cref{H2} that $h(S) = q_0$. Now, 
		\begin{align*}
		\gamma &= \alpha + q(s) & \text{by \cref{eq:betagamma}} \\
		&\leq \alpha + \frac{q_0}{2} & \text{since $s > \frac{r_0+r_1}{2}$ by \Cref{G6}} \\
		& < q_0 & \text{since $\alpha < \frac{q_0}{2}$ by \Cref{G6}}
		\end{align*}
		Thus, $h(S) = q_0$ exceeds $\gamma$, and so $S \subset \suplevel{h}{\gamma}$ as desired.
		
		Finally, to see that $S$ is the maximal invariant subset of $K$, begin with the facts $S \subset \GM{N}$ and $\GM{N} \subset \sublevel{g}{r}$ established above, so we have
		\begin{align*}
		S &\subset \GM{N} \cap \crit{f} \subset \sublevel{g}{r} \cap \crit{f}.
		\end{align*}
		\Cref{G3} guarantees $\sublevel{g}{r} \cap  \crit{f} = S$, so we conclude that $S = \GM{N} \cap \crit{f}$. Since $S$ lies on a single level set $f=0$, there are no connecting orbits between points of $S$, and so $S$ is the maximal invariant subset of $K$.
	\end{proof}

	\begin{lemma} \label{lem:IP23} $\GM{N}_-$ is positively invariant in $\GM{N}$ and is an exit set of $\GM{N}$, satisfying \Cref{IP2,IP3} respectively.
	\end{lemma}
	\begin{proof}
		We note from \cref{eq:gradflow} that $\sigma$ flows along the gradient $-\nabla{f}$, so $f$ is non-increasing along the flow. Thus if $x \in \GM{N}$, then $f\qty(\sigma(t,x)) \leq \alpha$; and by \cref{H4}, $h$ is also non-increasing along the flow.
		Thus if $x \in \GM{N}_-$, then $\gamma \geq h\qty(\sigma(t,x))$.
		Since $\GM{N}_- = \sublevel{f}{\alpha} \cap \intlevel{h}{\beta}{\gamma}$, if $x \in \GM{N}_-$, then any $\sigma(t,x) \in \GM{N}$ is in $\GM{N}_-$ if $t > 0$, therefore it is positively invariant in $\GM{N}$ and satisfies \Cref{IP2}.
		
		We now show that $\GM{N}_-$ satisfies the exit set condition \Cref{IP3}. Consider any $x \in \GM{N}$, such that $\sigma(t,x) \notin \GM{N}$ for some $t > 0$. Then either $f(\sigma(t,x)) > \alpha$ or $h(\sigma(t,x)) < \beta$. As $f(x) \leq \alpha$, and $f$ cannot increase along $\sigma(t,x)$, then we must have $h(\sigma(t,x)) < \beta$. As $h(x) \geq \beta$, there must be some $s \in [0,t)$ where $h(\sigma(s,x)) = \beta$ by continuity. Since $f$ cannot increase along $\sigma(t,x)$, we have $f(\sigma(s,x)) \leq \alpha$. Therefore, there is some $s \in [0,t]$  such that $\sigma(s,x) \in \GM{N}_- = \sublevel{f}{\alpha} \cap \suplevel{h}{\beta}$ for any $x \in \GM{N}$ that flows outside $\GM{N}$ at some $t> 0$.
	\end{proof}
	
	\section{The Geometry of Regular Intersections}
	\label{sec:geometry}
	Given our definition of $(\GM{N},\GM{N}_-)$ in \cref{eq:GMM}, the problem of inferring Conley indices is subsumed by the more general task of inferring the homology of subsets generated by taking finite intersections of level and sublevel sets of smooth functions at regular values. We parametrise this class of subsets as follows. 
	
	\begin{definition} \label{def:regcap} A non-empty subset $E$ of a compact Riemannian manifold $M$ is called a $(\mu, \Lambda)$-{\bf regular intersection} for real numbers $\mu > 0$ and $\Lambda \geq 0$ if there exist (finitely many) smooth functions $f_1, \ldots, f_\ell:M \to \R$, such that $E$ can be written as a finite intersection
		\begin{equation}
		E = \bigcap_{i =1}^\ell \level{f_i}{I_i} \label{eq:regcap}
		\end{equation}
		where each $\level{f_i}{I_i}$ is either a level set with $I_i = \set{0}$ or a sublevel set with $I_i = (-\infty, 0]$, with $0$ being a regular value of each $f_i$; moreover, these $f_i$ satisfy the following criteria:
		\begin{enumerate}[start = 1,label={\bfseries \upshape (R\arabic*)}]
			\item \label{def:regcap1} For any $1 \leq k \leq \ell$ and set of indices $1 \leq i_1 < i_2 < \cdots < i_k \leq \ell$ with $f_{(i_1, \ldots, i_k)}= (f_{i_1}, \ldots , f_{i_k}): M \to \R^k$, the Jacobian $\dd_p{f_{(i_1, \ldots, i_k)}}$ is surjective at all points $p$ in the intersection $\level{f_{(i_1, \ldots, i_k)}}{0} \cap E$, and the smallest non-zero singular value of this Jacobian  is greater or equal to $\mu$.
			\item \label{def:regcap2} The supremum $\sup_{p \in M} \|H_pf_i\|$ of the norm of each $f_i$'s Hessian on $M$ is bounded above by $\Lambda$; here
			\begin{equation}
			\norm{H_pf_i} := \sup_{\|X\|_{\R^m} = 1} \norm{H_pf_i(X)}_{\R^m}.
			\end{equation}
		\end{enumerate}
	\end{definition}
	
	Next we show that regular intersections are topologically well-behaved; in the statement below, we write $\interior{A}$ to indicate the interior of a subset $A$ of $M$, and $\closure{A}$ to denote its closure in $M$.

	\begin{lemma} \label{lem:regcapisreg} Every regular intersection of the form
		\begin{equation*}
		E = \bigcap_{i=1}^\ell \sublevel{f_i}{0}
		\end{equation*}
		is a regular closed subset of $M$. In particular, we have $E = \closure{\interior{E}}$ where
		\begin{align*}
		\interior{E} = \bigcap_{i=1}^\ell \sublevelopen{f_i}{0} .
		\end{align*}
		\begin{proof}
			We first check that $\interior{E}$ has the desired form; to this end, note that for each sublevel set $\sublevel{f_i}{0}$ taken at a regular value, the open sublevel set $\sublevelopen{f_i}{0}$ is the interior of $\sublevel{f_i}{0}$ (see \cite[Proposition 5.46]{lee_introduction_2012}). Thus, as $\sublevelopen{f_i}{0}$ is the largest open set in $\sublevel{f_i}{0}$, the intersection $\bigcap_{i=1}^\ell\sublevelopen{f_i}{0}$ must contain any open set of $E$, including $\interior{E}$. However, as $\sublevelopen{f_i}{0} \subset E$,
			\begin{align*}
			\interior{E} \subseteq \bigcap_{i=1}^l\sublevelopen{f_i}{0} \subset E \implies  \interior{E} = \bigcap_{i=1}^l\sublevelopen{f_i}{0}.
			\end{align*}
			Since it follows from the definition of closure that $E \supseteq \closure{\interior{E}}$, it suffices to show that $E \subseteq \closure{\interior{E}}$. Consider $p \in E \setminus \interior{E}$, where without loss of generality we assume that $f_1(p) = \cdots = f_k(p) = 0$ and $f_{i}(p) < 0$ for $i > k$. Let $u_i = \normvec{\nabla_p f_i }$. Let $\tilde{u}_i$ be the component of $u_i$ orthogonal to all $u_j$ where $j \neq i$ and $1 \leq j \leq k$. Since $E$ is a regular intersection (\Cref{def:regcap}), we have $\tilde{u}_i \neq 0$ and $\innerprod{\tilde{u}_i}{u_j}$ is positive and non-zero if and only if $i = j$. Define
			\begin{equation}
			v = -\sum_{i=1}^k \tilde{u}_i,
			\end{equation}
			so $\innerprod{v}{u_i} < 0$ for all $i= 1,\ldots, k$ (thus patently $v \neq 0$). Consider a continuous curve $\gamma(t)$ on $M$ where $\gamma(0) = p$ and $\dot{\gamma}(0) = v$. As $f_i$ are continuous, $f_i(p) = 0$ and $\innerprod{\nabla_p f_i }{\dot{\gamma}(0) } < 0$ for $i \leq k$; and $f_{i}(p) < 0$ for $i > k$, there is some sufficiently small $\epsilon > 0$ such that for all $t \in (0, \epsilon)$, we have $f_i(\gamma(t)) < 0$ for all $i \in \set{1,\ldots, n}$. We thus have for any $p \in E \setminus \interior{E}$ a sequence of points $\gamma(-t)$ for $t \in (0, \epsilon)$ in $E$, whose limit is $p$. Therefore,  $E \subseteq \closure{\interior{E}}$ as desired. 
		\end{proof}
	\end{lemma}
	
	We now proceed to analyse the geometry of regular intersections through the perspective of \cite{federer_curvature_1959}. 
	For any closed subset $A \subset \R^d$, let $d_A(x) := d_{\R^d}(x,A)$ denote the distance of any point $x \in \R^d$ to $A$, and let $\Nneighs{A}{x} \subseteq A$ be the set of nearest neighbours of $x$ in $A$. As $A$ is closed, $\Nneighs{A}{x}$ is a non-empty closed subset of $\R^d$. We let $\UP{A}$ be the set of points $x \in \R^d$ for which admit a unique nearest neighbour in $A$, and we call the complement of $\UP{A}$ in $\R^d$ the \emph{medial axis} of $A$:
	\begin{align}
	\UP{A} &:= \sett{x \in \R^d}{\# \Nneighs{A}{x} = 1}, \label{eq:UP} \\
	\med{A} &:= \sett{x \in \R^d}{\# \Nneighs{A}{x} > 1} = \R^d \setminus \UP{A}. \label{eq:med}
	\end{align}
	There is a projection map $\xi_A: \UP{A} \twoheadrightarrow A$ that sends each $x$ to its unique nearest neighbour in $A$. For $p$ in $A$, define the subset $\UP{A,p} = \sett{x \in \UP{A}}{\projmap{A}{x} = p}$. The \emph{local feature size} of $p \in A$ is 
	\begin{align}
	\lfs{A}{p} := d_{\R^d}(p, \med{A}). \label{eq:lfs}
	\end{align}
	We say $A$ has positive local feature size if $\lfs{A}{p} > 0$ for all $p \in A$. 
	
	\begin{definition}\label{def:reach}
		The {\bf reach} of $A$ is the infimum of the local feature size over $A$
		\begin{align}
		\reach_A := \inf_{p\in A} \lfs{A}{p}.
		\end{align}
		We say that $A$ has {\bf positive reach} if $\reach_A > 0$.
	\end{definition}
	
	One can also show that
	\begin{align}
	\reach_A &= \inf_{x \in \med{A}} d_A(x). \label{eq:reachmed}
	\end{align}
	Closed submanifolds of Euclidean space have positive reach \cite{lee_introduction_2012}, but in general the class of positive-reach subsets of $\R^d$ includes many non-manifold spaces. Our goal in this Section is to prove the following result, which is Lemma (B) from the Introduction.
	\begin{restatable}{lemma}{regpairreach}
		\label{lem:regreachbound} Every $(\mu, \Lambda)$-regular intersection $E = \bigcap_{i=1}^\ell \level{f}{I_i}$
		has its reach bounded from below by $\rho_k > 0$, which is given by
		\begin{equation}
		\frac{1}{\rho_k} = \frac{1}{\reach_M} + \sqrt{k} \cdot \frac{\Lambda}{\mu} \tag{\cref{eq:rhok}}
		\end{equation}
		where for any $p \in E$, the number of functions $f_i$ which is zero on $p$ is at most $k \leq l$.
	\end{restatable}
	\noindent In order to arrive at this result, we first recall some fundamental facts about the reach. 
	
	\subsection{Geometric Consequences of the Reach}
	
	If $A \subset \R^d$ has positive reach, then every $x \in \R^d$ with $d(x,A) < \reach_A$ has a unique projection $\xi_A(x)$ in $A$. This is expressed in Federer's tubular neighbourhood theorem \cite{federer_curvature_1959}, recalled below.
	
	\begin{pheorem} \label{thm:fedtube} Let $A$ be a subset of $\R^d$ with  $\reach_A >0$. Then for $r \leq \reach_A$, the set
		\begin{equation*}
		A^r := \bigcup_{p \in A} \ball{r}{p}
		\end{equation*}
		is entirely contained within $\UP{A}$ from \cref{eq:UP}.
	\end{pheorem}
	
	The reach also places constraints on the length of shortest paths between two points on a shape; here is the content of \cite[Theorem 1 \& Corollary 1]{boissonnat_reach_2019}. In the statement below, $\ballc{r}{x}$ denotes a closed Euclidean ball around a point $x$ whereas $\ball{r}{x}$ denotes the corresponding open ball.
	
	\begin{pheorem} \label{thm:geomdistortion} Assume that $A \subset \R^d$ has positive reach. Consider points $p,q \in A$ contained  $\ballc{r}{x}$ with  $r < \reach_A$. Then:
		\begin{enumerate}[label={\textup{(\roman*)}}]
			\item There is a geodesic path connecting $p$ and $q$ in $A$ which lies entirely within $\ball{r}{x} \cap A$.
			\item The length of this geodesic path is bounded above by
			\begin{align}
			d_A(p,q) \leq 2\reach_A \asin(\frac{\norm{p-q}}{2\reach_A}).
			\end{align}
		\end{enumerate}
	\end{pheorem}

	For $A \subset \R^d$ closed, consider the function $\fakereach_A : A \to [0, \infty]$ defined by
	\begin{align}
	\fakelfs{A}{p} &=
	\begin{cases}
	d_{\R^d}(p, \med{A,p}) &  \qif \med{A,p} \neq \emptyset \\
	\infty & \qif \med{A,p} = \emptyset.
	\end{cases} \label{eq:faketau}
	\end{align}
	This function furnishes lower bounds for the reach in the following sense. 
	\begin{lemma} \label{lem:faketau0} Let $A \subset \R^d$ be a closed subset. Then:
		\begin{enumerate}[ref=\thelemma (\roman*), label={\textup{(\roman*)}}]
			\item \label{lem:faketau0i}For any $p \in A$, we have $\fakelfs{A}{p}\geq \lfs{A}{p}$; and
			\item \label{lem:faketau0ii}$\inf_{p \in A} \fakelfs{A}{p} = \reach_A$.
		\end{enumerate}
		\begin{proof}  If $\med{A,p} = \emptyset$, then it follows from \cref{eq:faketau} that $\fakelfs{A}{p} = \infty \geq \lfs{A}{p}$. Otherwise, $\med{A}$ is non-empty since it contains $\med{A,p}$. Therefore,
			\begin{equation*}
			\fakelfs{A}{p} = d_{\R^d}(p, \med{A,p}) \geq d_{\R^d}(p, \med{A})  = \lfs{A}{p}.
			\end{equation*}
			We turn now to the second assertion. For $x \in \med{A}$, choose some point $p_x \in A$ such that $p_x \in \Nneighs{A}{x}$. From the definition of $\fakereach_A$ in \cref{eq:faketau}, we have $d(x,A) \geq \fakelfs{A}{p_x}$. Thus
			\begin{align*}
			\reach_A &= \inf_{x \in \med{A}} d(x,A) \geq \inf_{x \in \med{A}} \fakelfs{A}{p_x} \geq \inf_{p \in A} \fakelfs{A}{p}.
			\end{align*}
			As we have shown above that $\fakelfs{A}{p}\geq \lfs{A}{p}$, we also have an inequality in the opposite direction: $\inf_{p \in A} \fakelfs{A}{p}  \geq \reach_A$. Combining these two inequalities, we obtain $\reach_A = \inf_{p \in A} \fakelfs{A}{p}$.
			
		\end{proof}
	\end{lemma}
	
	We have considered $\fakelfs{A}{p}$ rather than the local feature size due to a convenient geometric property \cite[Theorem 4.8(7)]{federer_curvature_1959}.
	\begin{lemma} \label{lem:angletonormal} Let $A$ be a closed subset of $\R^d$ and consider $x \in \UP{A,p}$. Then for any $q \in A$,
		\begin{equation}
		\frac{\norm{q-p}}{2\fakelfs{A}{p}}>  \innerprod{\normvec{q-p}}{\normvec{x-p}}. \label{eq:angletonormal}
		\end{equation}
	\end{lemma}
	The geometric implications of this inequality are illustrated in \Cref{fig:fakelfs}. 
	
	\begin{figure}
		\centering
		\includegraphics[width = 0.6\textwidth]{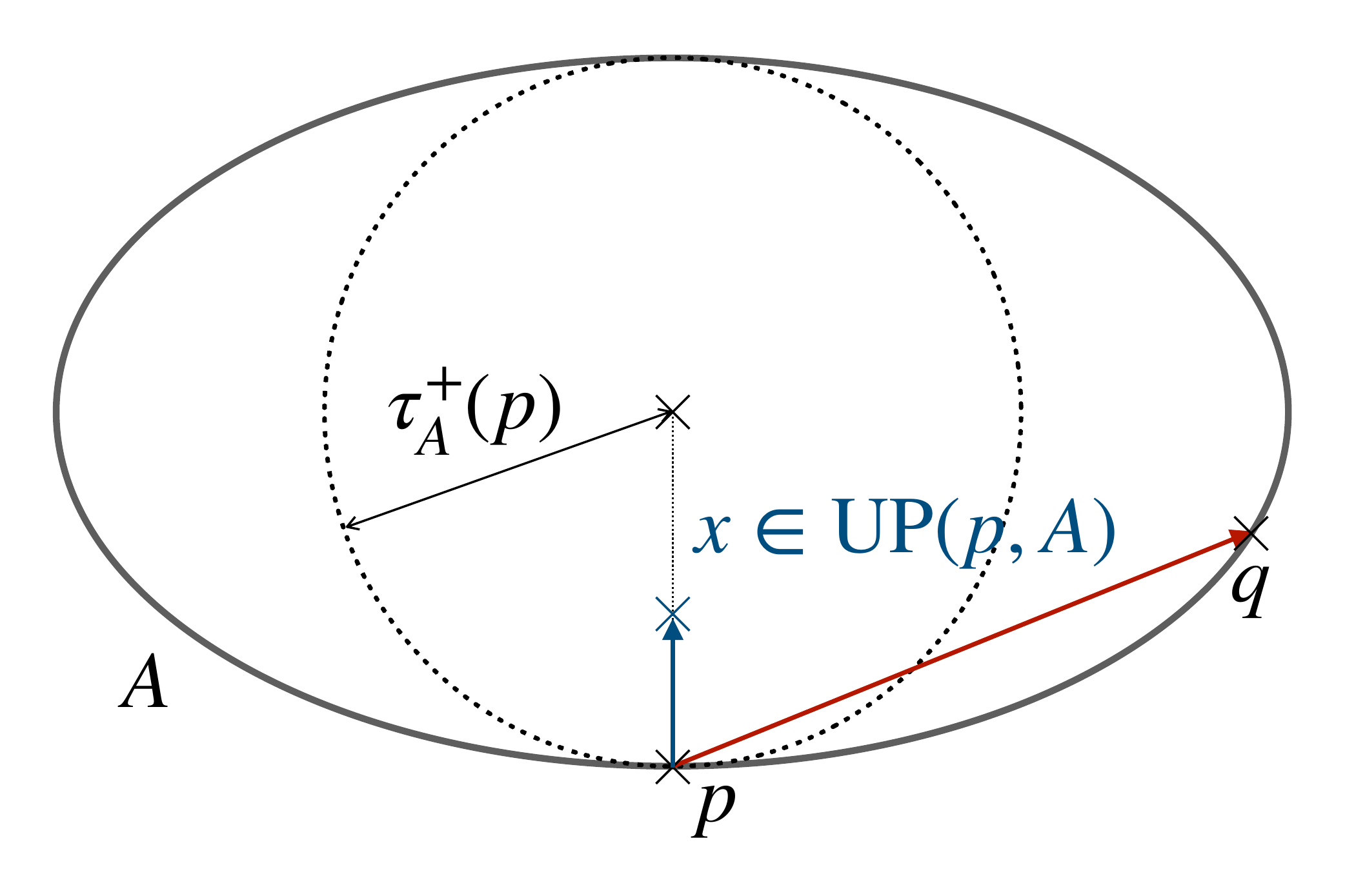}
		\caption{An illustration of the constraint placed on points $q \in A$ in relation to point $p \in A$ due to \Cref{lem:angletonormal}: implies any point $q \in A$ must lie outside the ball of radius $\fakelfs{A}{p}$ (boundary indicated by a dashed line), if the radial vector from $p$ to the centre of the ball points along $x-p$ for some $x \in \UP{A,p}$.}
		\label{fig:fakelfs}
	\end{figure}

	\subsection{The Reach of Manifolds}
	
	Here we collect some relevant facts about the reach of closed submanifolds of Euclidean space. 
	
	Fix a closed $m$-dimensional submanifold $M \subset \R^d$, and let $\tp{p}{M} \subset T_p\R^d$ be the plane tangent to $M$ at $p$ in the ambient Euclidean space. By $\zeta_p: M \to \tp{p}{M}$ we denote the restriction to $M$ of the orthogonal projection $\R^d \twoheadrightarrow \tp{p}{M}$. The normal space $N_pM$ to $M$ at $p$ is the kernel of this projection, i.e., the $(d-m)$-dimensional orthogonal complement to $\tp{p}{M}$ in $\R^d$. For any non-zero vectors $u \in T_pM$ and $v \in T_qM$ (where $p$ is not necessarily the same point as $q$), we let $\anglegeo{u}{v}$ be the angle between $u$ and $v$ parallel transported in the ambient Euclidean space to $T_0\R^d$ where $0$ is the (arbitrarily chosen) origin. Here is \cite[Theorem 4.8(12)]{federer_curvature_1959}.
	
	\begin{pheorem} \label{thm:manifoldtube} For any $p \in M$ and $r < \lfs{M}{p}$, we have
		\begin{equation}
		\UP{M,p} \cap \ball{r}{p} = \qty( p + N_pM) \cap \ball{r}{p}.
		\end{equation}
	\end{pheorem}
	
	\noindent Below we have reproduced \cite[Lemma 5 \& Corollary 3]{boissonnat_reach_2019}.
	
	\begin{lemma} \label{lem:paralleldistortion} For any $p$ and $q$ in $M$ with $\norm{p-q} \leq 2\reach_M $, consider a geodesic $\gamma: [0,s] \to M$ (given by \Cref{thm:geomdistortion}) with $\gamma(0) = p$ and $\gamma(s) = q$. Let $v(t) \in T_{\gamma(t)}M$ be the parallel transport of a unit vector $v \in T_{\gamma(0)}M$ along $\gamma$ to $T_{\gamma(t)}M$.  Then
		\begin{align}
		\anglegeo{v(0)}{v(t)} &\leq \frac{t}{\reach_M} \qand \\
		\sin(\frac{\anglegeo{v(0)}{v(t)} }{2}) &\leq \frac{\norm{\gamma(t) - \gamma(0)}}{2\reach_M}.
		\end{align}
	\end{lemma}
	
	\noindent And finally, the following result is \cite[Proposition 6.1]{niyogi_finding_2008} 
	
	\begin{lemma} If $M \subset \R^d$ is a compact Riemannian submanifold, then for any $p\in M$ and unit vectors $u,v$ in $T_pM$, the operator norm of the second fundamental form $\twoFF: T_pM \times T_pM \to N_pM$  is bounded above by \label{lem:twoformbound}
		\begin{equation}
		\norm{\twoff{u}{v}} \leq \frac{1}{\reach_M},
		\end{equation}
	\end{lemma}
	
	\subsection{Projection onto Tangent Planes}
	
	Let $M$ be a smooth closed submanifold of $\R^d$. For each point $p \in M$, we write $B_r(p)$ to indicate the Euclidean ball of radius $r > 0$ around $p$, and $\zeta_p:M \to T_pM$ to indicate the restriction of the orthoginal projection from $\R^d$ onto $T_pM$. The following result is a variant of \cite[Lemma 5.4]{niyogi_finding_2008}. 
	
	\begin{proposition}\label{prop:localdiff} For each $p \in M$ and $r < \sqrt{2} \reach_M$, the restriction of $\zeta_p$ to $\ball{r}{p} \cap M$ is a local diffeomorphism.
	\end{proposition}
	\begin{proof} It suffices to show that the Jacobian $\dd{{}_q\zeta_p}$ is injective at all $q \in \ball{r}{p} \cap M$, as the dimensions of the domain and codomain are the same (see \cite[Proposition 4.8]{lee_introduction_2012}).
		
		Suppose $\dd{\zeta_p}$ is singular at some $q \in \ball{r}{p}$. Then there is some unit vector $u \in T_qM$, when parallel transported in the ambient Euclidean space to $T_p\R^d$ along the line segment $\overline{qp}$,  is orthogonal to $T_pM \subset T_p\R^d$. As $\norm{p-q} < 2\reach_M$,  we can apply \Cref{thm:geomdistortion} and infer that there is a geodesic $\gamma: [0,s] \to M$ where $\gamma(0) = q$ and $\gamma(s) = p$. Let $v \in T_pM$ be the parallel transport of $u \in T_qM$ along this geodesic.
		
		As $u \in N_pM$, it is orthogonal to $v \in T_pM$ and we have $\anglegeo{u}{v} = \frac{\pi}{2}$. Applying the bound for $\anglegeo{u}{v}$ in \Cref{lem:paralleldistortion}, we have
		\begin{equation*}
		\frac{t}{\reach_M} \geq \frac{\pi}{2}.
		\end{equation*}
		Substituting this into \Cref{thm:geomdistortion}, we obtain
		\begin{equation*}
		r \geq 2\reach_M \sin(\frac{t}{2\reach_M}) \geq 2\reach_M \sin(\frac{\pi}{4}) = \sqrt{2}\reach_M.
		\end{equation*}
		which contradicts our assumption that $r < \sqrt{2}\reach_M$; thus, $\dd{\zeta_p}$ is injective at $q$ as desired.
	\end{proof}
	
	Next we prove a slight extension of this result. 
	\begin{figure}
		\centering
		\begin{subfigure}[b]{0.45\textwidth}
			\includegraphics[width = 0.85\textwidth]{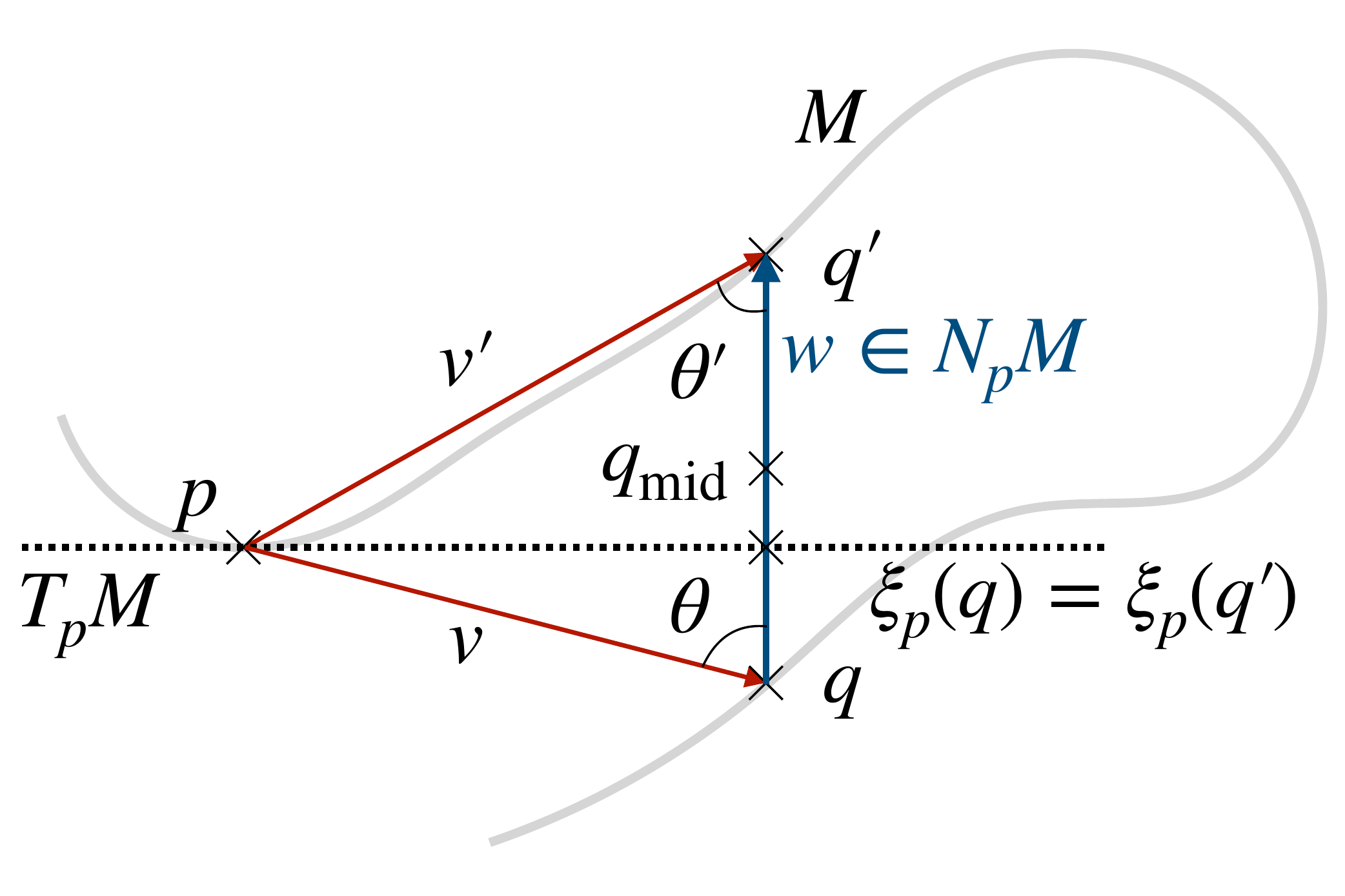}
			\caption{Illustration of a geometric argument in deriving \Cref{eqstep:tproj2} in the proof of \Cref{prop:ballembed}. N.B. for ambient dimension $d \geq 2$, $\zeta_p(q)$ need not lie on the line joining $q$ and $q'$. }
			\label{fig:eqstep:tproj2}
		\end{subfigure}\hfill
		\begin{subfigure}[b]{0.45\textwidth}
			\includegraphics[width = 0.85\textwidth]{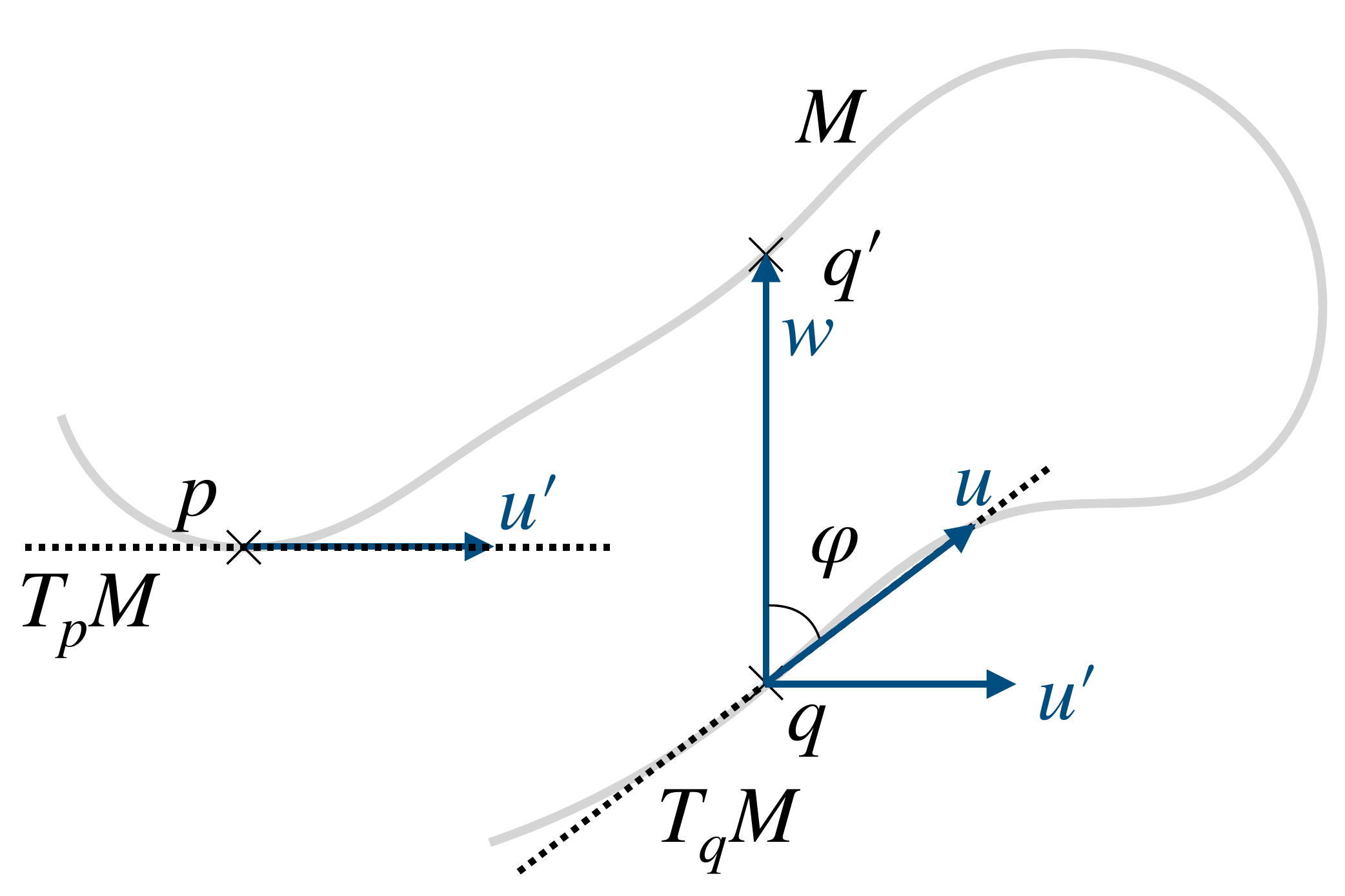}
			\caption{Illustration of a geometric argument in deriving \Cref{eqstep:tproj3} in the proof of \Cref{prop:ballembed}.}
			\label{fig:eqstep:tproj3}
		\end{subfigure}
		\caption{Geometric illustrations for the proof of \Cref{prop:ballembed}.}
	\end{figure}
	
	\begin{restatable}{proposition}{ballembed}
		\label{prop:ballembed}
		For all $p \in M$, the map $\zeta_p$ is a smooth embedding of $\ball{\reach_M}{p} \cap M$ into $\tp{p}{M}$.
	\end{restatable}
	\begin{proof} Consider $r < \sqrt{2}\reach_M$. It sufficies to show that $\zeta_p$ restricted to $\ball{r}{p} \cap M$ is a smooth immersion (i.e. $\dd{\zeta_p}$ is injective), that $\zeta_p$ is an open map, and that $\zeta_p$ is injective (see \cite[Proposition 4.22]{lee_introduction_2012}). As $\zeta_p$ is a local diffeomorphism on $\ball{r}{p}$ for  $r < \sqrt{2}\reach_M$ (see \Cref{prop:localdiff}), it is an open map (\cite[Proposition 4.6]{lee_introduction_2012}). Thus all that remains is to show that $\zeta_p$ is injective for $r$ sufficiently small.
		
		Suppose $\zeta_p$ is not injective on $\ball{r}{p} \cap M$ and consider the illustration in \Cref{fig:eqstep:tproj2}.  Suppose there are two distinct points $q$ and $q'$ in $\ball{r}{p} \cap M$ that project onto the same point in $\zeta_p(q) = \zeta_p(q') \in \tp{p}{M}$. Since \Cref{thm:manifoldtube} implies $q - \zeta_p(q)$ and $q'- \zeta_p(q')$ are vectors in $N_pM$, the vector $w = q - q' = (q - \zeta_p(q)) - (q' - \zeta_p(q')$ is also a vector in $N_pM$.
		
		As a shorthand, let $v = q-p$ and $v'=  q'-p$. Consider then $\theta = \anglegeo{w}{-v}$ and $\theta' = \anglegeo{w}{v'}$; applying \Cref{lem:angletonormal},
		\begin{align} \label{eqstep:tproj1}
		\cos(\theta) = \innerprod{\normvec{-v}}{\normvec{w}} \leq \frac{\norm{v}}{2\reach_M} < \frac{r}{2\reach_M},
		\end{align}
		and similarly, $\cos(\theta') < \frac{r}{2\reach_M}$. As $r  <\reach_M$, angles $\theta$ and $\theta'$ must be strictly greater than $\frac{\pi}{3}$.
		Consider then the triangle formed by $p$, $q$ and $q'$. Applying the cosine rule, we have
		\begin{align*}
		\norm{w}^2 &= \norm{v}^2 + \norm{v'}^2 - 2 \norm{v}\norm{v'} \cos(\pi -( \theta + \theta')) \\ &< 2r^2(1+\cos(\theta+\theta')) \\
		& = 4r^2 \cos[2](\frac{\theta + \theta'}{2}).
		\end{align*}
		The inequality is due to the bound $q$ and $q'$ being in $\ball{r}{p}$, whence $\norm{v} = \norm{q-p}$ and $\norm{v'} = \norm{q' -p}$ are both strictly less than $r$. Maximising the right hand side by choosing $\theta$ and $\theta'$ to be as small as possible subject to the constraints of \cref{eqstep:tproj1}, we obtain
		\begin{align}
		\norm{w} <\frac{r^2}{\reach_M}. \label{eqstep:tproj2}
		\end{align}
		Suppose now $w \in N_qM$ and consider the point $q_\text{mid} = q + \frac{w}{2} = \frac{q+q'}{2}$. Since we have assumed $r < \sqrt{2}\reach_M$, \cref{eqstep:tproj2} implies $\norm{w} < 2\reach_M$, and thus $\norm{q_\text{mid}-q} < \reach_M$. Since we have assumed  $(q_\text{mid} - q) \propto w \in N_qM$,  \Cref{thm:manifoldtube} implies  $q_\text{mid}$ has a unique projection onto $M$, which is $q$. However, $q_\text{mid}$ is equidistant to both $q$ and $q'$ which lie in $M$; hence we reach a contradiction and $w$ must thus have a non-zero component in $T_qM$ (as a vector subspace of $T_q\R^d$).
		
		Consider the illustration in \Cref{fig:eqstep:tproj3}. Let $u$ be the non-zero projection of $w$ onto $T_qM$ and let $\phi = \anglegeo{u}{w}$. If $w = u$, then $\phi = 0$; else, if $w - u \neq 0$, $w$ has a component $w-u$ in $N_qM$. Thus we can apply \Cref{lem:angletonormal} once again and obtain a
		\begin{align}
		\sin(\phi) &= \cos(\frac{\pi}{2} - \phi)  = \innerprod{\normvec{w}}{\normvec{w-u}} \leq \frac{\norm{w}}{2\reach_M} < \frac{r^2}{2\reach_M^2} &  \label{eqstep:tproj3}
		\end{align}
		where we substituted \cref{eqstep:tproj2} in the final line.
		
		As $\norm{p-q} < \sqrt{2}\reach_M$, there is a geodesic $\gamma$ on $M$ such that $\gamma(0) = q$ and $\gamma(s) = p$ due to \Cref{thm:geomdistortion}. Let us parallel transport $u \in T_qM$ along $\gamma$ to $T_pM$. Let $u'$ be the transported vector in $T_pM$. Applying \Cref{lem:paralleldistortion},
		\begin{equation*}
		\sin(\frac{\anglegeo{u}{u'}}{2}) \leq \frac{r}{2\reach_M} \implies \cos(\anglegeo{u}{u'}) \geq 1-\frac{1}{2}\qty(\frac{r}{\reach_M})^2.
		\end{equation*}
		As $r  < \sqrt{2} \reach_M$, the right hand side is positive and hence $\anglegeo{u}{u'}< \frac{\pi}{2}$.
		
		The triangle inequality on $\S^{d-1}$ implies
		\begin{equation}
		\phi = \anglegeo{u}{w} \geq \anglegeo{u'}{w} - \anglegeo{u}{u'} = \frac{\pi}{2} - \anglegeo{u}{u'} \label{eqstep:tproj4}
		\end{equation}
		where the last equality is due to $u' \in T_pM$ and $w \in N_pM$. As $u$ is a non-zero projection of $w$ onto $T_qM$, the angle $\phi = \anglegeo{u}{w}$ between $u$ and $w$  is at most $ \phi < \frac{\pi}{2}$. In addition, as we have shown that  $\anglegeo{u}{u'}< \frac{\pi}{2}$, the sine function is montonic on both sides of the inequality. Hence, we can combine the inequalities of \Cref{eqstep:tproj3,eqstep:tproj4} and obtain the following:
		\begin{align*}
		\frac{r^2}{2\reach_M^2} &> \sin(\phi) \geq \sin(\frac{\pi}{2} - \anglegeo{u}{u'}) = \cos(\anglegeo{u}{u'})  \geq 1-\frac{1}{2}\qty(\frac{r}{\reach_M})^2. \\
		\implies r &> \reach_M.
		\end{align*}
		We have thus shown that any two points $q$ and $q'$ that project onto the same point in $T_pM$ must be at least $\reach_M$ away from $p$. Thus the projection is injective on $\ball{r}{p} \cap M$.
	\end{proof}

	\subsection{Bounding the Reach of Regular Intersections}
	
	It will be helpful to first focus on the cases where the regular intersection at hand is a level set or sublevel set of a single smooth function $f: M \to \R$ on a compact submanifold $M$ of $\R^d$. We suppose without loss of generality that $0$ is a {regular value} of $f$. Consequently, the level set $\level{f}{0}$ is a codimension-1 submanifold of $M$, and $\sublevel{f}{0}$ is a codimension-0 submanifold of $M$ with $\level{f}{0}$ as its boundary (see \cite{milnor1997topology}). Since $f$ is continuous, both sets $\level{f}{0}$ and $\sublevel{f}{0}$ are closed in $M$. As $M$ is compact in $\R^d$, it follows that these sets are also compact. Thus $\level{f}{0}$ is a compact submanifold of $\R^d$, and therefore $\level{f}{0}$ has positive reach.
	
	\begin{proposition} \label{lem:fsingle} Suppose $f:M \to \R$ is smooth on a positive reach closed submanifold $M \subset \R^d$. Consider $x \in \R^d$ and $p \in \level{f}{0}$. Assume $\norm{\nabla_p f} = \mu > 0$. Then:
		\begin{enumerate}[ref= Proposition \thelemma (\roman*), label={\textup{(\roman*)}}]
			\item \label{lem:fsingle1} If $p \in \Nneighs{\level{f}{0}}{x}$, then $x-p = n + \lambda \nabla_p f$ where $n \in N_pM$.
			\item \label{lem:fsingle2}Let $\Lambda$ be an upper bound on the norm of the Hessian of $f$ on $\ballc{\reach_M}{x} \cap M$ and
			\begin{equation}
			\frac{1}{\rho} =  \frac{\Lambda}{\mu}  +  \frac{1}{\reach_M}. \label{eq:rho}
			\end{equation}
			If $x-p = n + \lambda \nabla_p f$ and $\norm{x-p} < \rho$, then $\projmap{\level{f}{0}}{x} = p$.
			\item \label{lem:fsingle3} Consequently, we have
			\begin{equation}
			\UP{\level{f}{0},p} \cap \ball{\rho}{p} = \sett{x \in \ball{\rho}{p}}{ x-p= n+\lambda \nabla_p f \text{ and }  n \in N_pM},
			\end{equation}
			along with ${\fakelfs{\level{f}{0}}{p}} \geq \rho$.
			\item \label{lem:fsingle4} Consider some $x\in \R^d$ such that $\ball{r}{x} \cap M \neq \emptyset$, and $r < \reach_M$. If $\ballc{r}{x} \cap \level{f}{0} = p$ and $\ball{r}{x} \cap \level{f}{0} = \emptyset$,  then $f$ is either non-negative or non-positive on $\ball{r}{x} \cap M$.
		\end{enumerate}
		\begin{proof}
			Let $g(y) =\frac{\norm{x-y}^2}{2}$ and $\tilde{g}: M \to [0, \infty)$ be the restriction of $g$ to $M$. We note that $\nabla \tilde{g}$ is the projection of $x-p$ onto $T_pM$.
			\begin{enumerate}[label={\textup{(\roman*)}}]
				\item If $\nabla_p \tilde{g} = 0$, then $x-p \bot T_pM$; else, $\level{\tilde{g}}{r^2/2}$ is tangent to $\level{f}{0}$ at $p$, and $\nabla_p \tilde{g}$ is parallel with $\nabla_p f$. As $\nabla \tilde{g}$ is the projection of $x-p$ onto $T_pM$, we can thus write $x-p = n + \lambda \nabla_p f$ where $n \in N_pM$.
				\item As $\level{f}{0}$ is closed, $x$ must have a nearest neighbour in $\level{f}{0}$. Suppose $p \notin \Nneighs{\level{f}{0}}{x}$ and consider $q \in \Nneighs{\level{f}{0}}{x}$. Then $\norm{x-q} \leq \norm{x-p} = r$.  Suppose $r < \reach_M$. Then, $\ballc{r}{x} \cap M$ is connected by \Cref{thm:geomdistortion}. Thus there is a geodesic $\gamma :[0,s] \to M$ between $\gamma(0) = p$ and $\gamma(s) = q$ parametrised by arc length.
				
				Consider $\hat{f} = f \circ \gamma: [0,s]  \to \R$ and $\hat{g} = g \circ \gamma: [0,s]  \to \R$. Then by definition $\hat{f}(0) = \hat{f}(s)$ and $\hat{g}(0) \geq \hat{g}(s)$. If we Taylor expand $\hat{f}$ and $\hat{g}$, we have for some $s_1$ and $s_2$ in $(0,s)$
				\begin{align*}
				\hat{f}(s) &= \hat{f}(0) + \dv{\hat{f}}{t}\eval_{t=0} s +  \frac{1}{2}\dv[2]{\hat{f}}{t}\eval_{t=s_1} s^2  \implies \dv{\hat{f}}{t}\eval_{t=0}  +  \frac{1}{2}\dv[2]{\hat{f}}{t}\eval_{t=s_1} s = 0 \\
				\hat{g}(s) &= \hat{g}(0) + \dv{\hat{g}}{t}\eval_{t=0} s +  \frac{1}{2}\dv[2]{\hat{g}}{t}\eval_{t=s_2 } s^2  \implies \dv{\hat{g}}{t}\eval_{t=0} +  \frac{1}{2}\dv[2]{\hat{g}}{t}\eval_{t=s_2 } s \leq 0.
				\end{align*}
				Evaluating the derivatives, and substituting $x-p =  n + \lambda \nabla_p f$, we obtain
				\begin{align*}
				\innerprod{\vel(0)}{\nabla_p f}  +  \frac{1}{2}\innerprod{\vel(s_1)}{\nabla_{\vel (s_1)}\nabla f (\gamma(s_1))}  s &= 0  \qand \\
				\lambda \innerprod{\vel(0)}{\nabla_p f} +  \frac{1}{2} \qty(1+ \innerprod{\gamma(s_2) - x}{ \acc (s_2)})s &\leq 0, \\
				\implies 1 - \lambda \innerprod{\vel(s_1)}{\nabla_{\vel (s_1)}\nabla f (\gamma(s_1))}  + \innerprod{\gamma(s_2) - x}{ \acc (s_2)} &\leq 0.
				\end{align*}
				As $\ball{t}{x} \cap M$ is geoesically convex for $t < \reach_M$ and $\norm{x-q} \leq \norm{x-p} < \reach_M$, we thus have $\gamma([0,s]) \subset \ballc{r}{x}$. Hence $\norm{\gamma(s_2) - x} \leq r$. Applying the bound on $\norm{\acc}$ \Cref{lem:twoformbound} and the Hessian of $f$, we have
				\begin{equation*}
				{\abs{\lambda}}{\Lambda}  +  \frac{r}{\reach_M} \geq 1.
				\end{equation*}
				In addition, since $x-p = n + \lambda \nabla_p f$ and $n \perp \nabla_p f$, we have $\abs{\lambda} \leq \frac{r}{\mu}$. Thus the existence of $q \in \Nneighs{\level{f}{0}}{x}$ not equal to $p$, and  $\norm{x-q} \leq \norm{x-p} = r$ implies
				\begin{equation*}
				r \geq \qty(\frac{\Lambda}{\mu} +  \frac{1}{\reach_M})^{-1} =: \rho
				\end{equation*}
				Hence, if $r = \norm{x-p} < \rho$, the point $x$ must project onto $\level{f}{0}$ at $p$.
				\item As a consequence of \Cref{lem:fsingle1,lem:fsingle2}, any $x \in \ball{\rho}{p}$ that projects onto $\level{f}{0}$ at $p$ is a linear combination of some $n \in N_pM$ and $\nabla_p f$, and vice versa any $x \in \ball{\rho}{p}$ that is a linear combination of some $n \in N_pM$ and $\nabla_p f$ projects onto $\level{f}{0}$ at $p$.
				
				For $x \in \med{\level{f}{0}, {p}}$, \Cref{lem:fsingle1} implies $x-p = n + \lambda \nabla_p f$. As $x \notin \UP{\level{f}{0},p}$, $\norm{x-p} \geq  \rho$. Thus  ${\fakelfs{\level{f}{0}}{p}} \geq \rho$.
				
				\item  Since $r < \reach_M$, then \Cref{thm:geomdistortion} implies $\ballc{r}{x} \cap M$ is connected. We claim that $\ball{r}{x} \cap \level{f}{0} = \emptyset$ implies the sign of $f$ must be non-positive or non-negative on $\ballc{r}{x} \cap M$. Suppose that is not the case; then there are two points of opposite sign in $\ballc{r}{x} \cap M$. Consider a path connecting those two points contained in $\ball{r}{x} \cap M$. As $f$ is continuous, there must be some point along the path where $f$ is zero. However this contradicts $\ball{r}{x} \cap \level{f}{0} = \emptyset$.
			\end{enumerate}
		\end{proof}
	\end{proposition}
	The above observations about sublevel and level sets give us a stepping stone towards deriving the bound for the reach of regular intersections of multiple sublevel and level sets. 
	
	\begin{proposition} \label{lem:fmulti} Consider a $(\mu, \Lambda)$-regular intersection as in \Cref{def:regcap}:
		\begin{equation}
		E = \bigcap_{i=1}^\ell \level{f}{I_i} \tag{\Cref{eq:regcap}}
		\end{equation}
		For $0 \leq k \leq \ell$, define
		\begin{equation}
		\frac{1}{\rho_k} = \frac{1}{\reach_M} + \sqrt{k}\frac{\Lambda}{\mu}. \label{eq:rhok}
		\end{equation}
		Consider $x \in \R^d$ and $p \in E$. Without loss of generality, assume $f_i(p) = 0$ for $i \leq k$, and $f_i(p) <0$ otherwise. Then:
		\begin{enumerate}[ref=\thelemma (\roman*), label={\textup{(\roman*)}}]
			\item \label{lem:fmulti1}If $p \in \Nneighs{E}{x}$, then there are coefficients $\lambda_i$ and $n \in N_pM$, such that
			\begin{equation}
			x-p = n + \sum_{i = 1}^k \lambda_i\nabla_p f_i,   \label{eq:xpfmulti}
			\end{equation}
			where $\lambda_i \geq 0$ if $I_i = \lezero$.
			\item \label{lem:fmulti2}Conversely, if $x-p$ can be written in the form \cref{eq:xpfmulti}, and if
			$\norm{x-p} < \rho_k$, then we have $\projmap{E}{x} = p$.
			\item \label{lem:fmulti3}Consequently,
			\begin{equation}
			\UP{E,p} \cap \ball{\rho_k}{p} = \sett{x \in \ball{\rho_k}{p}}{x-p\  \text{satisfies \cref{eq:xpfmulti}}}.
			\end{equation}
			and ${\fakelfs{E}{p}} \geq \rho$.
		\end{enumerate}

		\begin{proof} Let $r = \norm{x-p}$ and consider the open Euclidean ball $\ball{r}{x}$.
			\begin{enumerate}[label={\textup{(\roman*)}}]
				\item If $\ball{r}{x} \cap M = \emptyset$, then $p \in \Nneighs{M}{x}$ and $x-p \in N_pM$. Else, consider the case $\ball{r}{x} \cap M \neq \emptyset$ and write $x-p = n + t+ \sum_{i = 1}^k \lambda_i\nabla_p f_i $, where $t \in T_pM$ is orthogonal to $\nabla_p f_1 , \ldots, \nabla_p f_k $.
				
				Let $\gamma: [0,1] \to M$ be a smooth curve wholly contained in $E$, with $\gamma(0) = p$ and $\dot{\gamma}(0) = \nu$. As $f_i(p) = 0$ for $i \leq k$, for $I_i = \lezero$
				\begin{align*}
				\exists \epsilon > 0 \qq{s.t.} f_i(\gamma(t)) < 0\quad  \forall t \in (0, \epsilon)  &\iff \innerprod{\nu}{\nabla_p f_i } < 0
				\end{align*}
				and for all $i$,
				\begin{align*}
				f_i(\gamma(t)) = 0  &\iff \innerprod{\dot{\gamma}(t)}{\nabla_{\gamma(t)} f_i} = 0
				\end{align*}
				As $p \in \Nneighs{E}{x}$, the ball $\ball{r}{x}$ cannot intersect any point in $E$; as such, we also have
				\begin{equation*}
				\innerprod{\nu}{x-p} \leq 0 \implies \innerprod{\nu}{t} + \sum_{i=1}^k \lambda_i \innerprod{\nu}{\nabla_p f_i } \leq 0 \tag{$\ast$}
				\end{equation*}
				where we substituted our expression for $x-p$ and noted that $n \in N_pM$ and $\nu \in T_pM$. Choosing $\gamma$ so that $f_i(\gamma(t)) = 0$ for $i \leq k$ and $\nu = t$, we substitute into ($\ast$) and deduce
				\begin{equation*}
				\innerprod{t}{t} \leq 0 \implies t = 0.
				\end{equation*}
				Then, as  $\nabla_p f_1 , \ldots, \nabla_p f_k $ are linearly independent, we can choose $\nu$ such that for one $j \leq k$ where $I_j = \lezero$, we have  $\innerprod{\nu}{\nabla_p f_j}< 0$ and $\innerprod{\nu}{\nabla_p f_i }= 0$ for $i \neq j$. Substituting this choice of $\nu$ into ($\ast$),
				\begin{equation*}
				\lambda_j \innerprod{\nu}{\nabla_p f_j} \leq 0 \qand  \innerprod{\nu}{\nabla_p f_j}< 0 \implies \lambda_j \geq 0.
				\end{equation*}
				Thus $x-p = n + \sum_{i = 1}^k \lambda_i\nabla_p f_i $ where $\lambda_i \geq 0$ if $I_i = \lezero$.
				\item First, let us consider the case where $\lambda_i = 0$. Then $x-p \in N_pM$. As  $r = \norm{x-p} < \lfs{M}{p}$, we can apply \Cref{thm:manifoldtube} and deduce that $\projmap{M}{x} = p$. As $E \subset M$, therefore $\projmap{E}{x} = p$.
				
				Then let us consider the case where $\lambda_i \neq 0$ for some $i$. Let $\tilde{f} = \sum_{i=1}^k \alpha_i f_i$ where $\alpha_i = \lambda_i /\norm{\lambda}$, where where $\norm{\lambda} = \sqrt{\sum_{i=1}^k \lambda_i^2}$.
				
				Since $f_i(q) \leq 0$ for $q \in E$ we note that $E\subset \sublevel{\tilde{f}}{0}$. In particular, $p \in \level{\tilde{f}}{0}$. Since $E$ is a regular intersection, we can check that $\norm{\nabla \tilde{f}} \geq \mu  > 0$ due to \Cref{def:regcap1}. Furthermore, due to \Cref{def:regcap2},  for any $q\in M$ and unit vector $X \in T_q M$,
				\begin{align*}
				\abs{\innerprod{X}{\nabla_X \nabla \tilde{f}}} &= \abs{ \sum_{i=1}^k \alpha_i  \innerprod{X}{\nabla_X \nabla f_i} }\\
				&\leq \Lambda \sup_{\norm{\alpha} = 1} \sum_{i=1}^k \alpha_i   = \sqrt{k}\Lambda.
				\end{align*}
				Thus $\tilde{\Lambda} = \sqrt{k}\Lambda$ is an upper bound on the Hessian of $\tilde{f}$. Because $(x-p) \propto (n + \nabla \tilde{f}(p))$ and
				\begin{equation*}
				r = \norm{x-p} < \rho_k = \qty(\frac{1}{\reach_M} + \frac{\tilde{\Lambda}}{\mu})^{-1},
				\end{equation*}
				we can apply \Cref{lem:fsingle} to $p$ in $\level{\tilde{f}}{0}$ and deduce that $\projmap{\level{\tilde{f}}{0}}{x} = p$.
				
				Consider $\ballc{r}{x} \cap M$. Since $\projmap{\level{\tilde{f}}{0}}{x} = p$, we have  $\ballc{r}{x} \cap \level{\tilde{f}}{0} = p$ and $\ball{r}{x} \cap \level{\tilde{f}}{0} = \emptyset$. Furthermore, since $r < \reach_M$, the sign of $\tilde{f}$ is either non-negative or non-positive on $\ballc{r}{x} \cap M$ according to \Cref{lem:fsingle4}. We now determine the sign of $\tilde{f}$ on $\ballc{r}{x} \cap M$. Since there is some $i$ for which  $\lambda_i \neq 0$, let us consider a smooth curve $\sigma(t)$ on $M$ where $\sigma(0) = p$ and $\dot{\sigma}(t) = \nabla_{\sigma(t)} \tilde{f}$. Since $\innerprod{\dot{\sigma}(0)}{x-p} \propto \norm{\nabla_p \tilde{f}}^2 > 0$,  the curve enters $\ballc{r}{x}$ at $p$. As $\tilde{f}$ increases away from $0$ along the curve and for sufficiently small $\epsilon > 0$, $\sigma(\epsilon) \in \ball{r}{x}\cap M$, we deduce that $\tilde{f} > 0$ on $\ball{r}{x}\cap M$ as we have shown that $\tilde{f}$ is of constant sign on $\ball{r}{x}\cap M$. In other words, $\sublevel{\tilde{f}}{0} \cap \ball{r}{x} = \emptyset$. As we have shown that $\ballc{r}{x} \cap \level{\tilde{f}}{0} = p$, the continuity of $\tilde{f}$ implies $\sublevel{\tilde{f}}{0} \cap \ballc{r}{x} = p$.
				
				Because $E \subset \sublevel{\tilde{f}}{0}$ and $\sublevel{\tilde{f}}{0} \cap \ballc{r}{x} = p$, we must therefore have $E \cap \ballc{r}{x} = p$ i.e. $\projmap{E}{x} = p$.
				\item Follows the reasoning of the proof of \Cref{lem:fsingle3}.
			\end{enumerate}
		\end{proof}
	\end{proposition}
	
	We have now arrived at the proof of \Cref{lem:regreachbound}.
	
	\begin{proof}[Proof of \Cref{lem:regreachbound}]
		As we have produced a bound on $\fakelfs{E}{p}$ for any $p \in E$, \Cref{lem:fsingle3}, this follows from applying \Cref{lem:faketau0ii} that $\reach_E = \inf_{p \in E} \fakelfs{E}{p} \geq \rho_k$ where $\rho_k$ is as defined in \cref{eq:rhok}.
	\end{proof}
	
	In the next subsection, we establish theoretical guarantees for recovering the homology of regular intersections via the more general homology inference results for subsets with positive reach. 
	
	\subsection{Homological Inference of Subsets with Positive Reach} \label{subsec:hominfposreach}
	
	It was shown in \cite{niyogi_finding_2008,wang_topological_2020} showed that a sufficiently dense point sample of compact manifolds with and without boundary can be used to recover their homotopy type. These arguments are readily generalised to the following inference result for Euclidean subsets with positive reach.
	
	\begin{restatable}{proposition}{nonoise}
		\label{prop:nonoise}
		Let $A$ be a subset of $\R^d$ with positive reach $\reach_A$. Suppose we have a finite point sample $\X \subset A$ that is $\delta$-dense in $A$ where $\delta < \frac{\bar{\reach}}{4}$, where $\bar{\reach} \in (0, \reach_A]$ is some positive lower bound on the reach of $A$. Then $\X^\epsilon$ deformation retracts to $A$ if
		\begin{equation}
		\frac{\epsilon}{\bar{\reach}} \in \qty(\frac{1}{2}- \sqrt{\frac{1}{4}-\frac{\delta}{\bar{\reach}}}, \frac{1}{2} + \sqrt{\frac{1}{4}-\frac{\delta}{\bar{\reach}}}).
		\end{equation}
	\end{restatable}

	Given the explicit lower bound for the reach of regular intersections in \Cref{lem:regreachbound}, the following homology inference result for regular intersections is a direct corollary of \Cref{prop:nonoise}.
	
	\begin{corollary}\label{cor:regintdense}
		Suppose we have a finite point sample $\X \subset E$ is $\delta$-dense in a regular intersection $E = \bigcap_{i=1}^l \sublevel{f}{I_i}$ as given in \cref{eq:regcap}. Let $k$ be the maximum number of level sets $\level{f_i}{0}$ that intersect in $E$. If $\delta < \frac{\rho_k}{4}$, where $\rho_k$ is given by \cref{eq:rhok}, then $\X^\epsilon$ deformation retracts to $A$ if
		\begin{equation}
		\frac{\epsilon}{\rho_k} \in \qty(\frac{1}{2}- \sqrt{\frac{1}{4}-\frac{\delta}{\rho_k}}, \frac{1}{2} + \sqrt{\frac{1}{4}-\frac{\delta}{\rho_k}}).
		\end{equation}
	\end{corollary}
	We devote the remainder of this subsection to deriving \Cref{prop:nonoise}, closely following the original argument for manifolds in \cite[Proposition 4.1]{niyogi_finding_2008}.
	
	\begin{lemma}\label{lem:defretr} Let $A$ be a subset of $\R^d$ with positive reach. Suppose for some $B \subset \R^d$, we have $A \subset B \subset \UP{A}$. If $B$ is star-shaped relative to $A$ -- i.e. if for every $p \in B$, the line segment $\overline{p\xi_A\qty(p)}$ is also contained in $B$ -- then $B$ deformation retracts to $A$.
	\end{lemma}
	\begin{proof}
		Let us consider the function $F: [0,1] \times B \to \R^d$
		\begin{equation*}
		F(t,p) = (1-t)\cdot p + t\cdot \xi_A(p).
		\end{equation*}
		Since projection maps $\xi_A$ are continuous \cite[Theorem 4.8(2)]{federer_curvature_1959}, the map $F$ is continuous. Furthermore, if $\overline{p\xi_A(p)}$ is contained in $B$, then $F(t,p) \subset B$. We can easily check that $F(0,p)=p$,  $F(1,p) = \xi_A(p) \in A$, and $F(t,p) = \xi_A(p) = p$. Therefore, $F$ is a strong deformation retraction $F : [0,1] \times B \to B$ of $B$ onto $A$.
	\end{proof}

	\begin{lemma} \label{lem:nonoiseoffsetbound}  Assume the conditions of \Cref{prop:nonoise}. Consider $x \in \X$ and  $p \in \ball{\epsilon}{x} \cap \UP{A,q}$ where $\epsilon \in (0,\reach_A)$. If $q \notin \ball{\epsilon}{x}$, then there is a unique point $y = \partial \ball{\epsilon}{x} \cap \overline{qp} $, such that
		\begin{equation*}
		\norm{y - q} \leq \frac{\epsilon^2}{\reach_A}.
		\end{equation*}
	\end{lemma}
	\begin{proof}

		Since $\epsilon < \reach_A$, and  $p \in \ball{\epsilon}{x}$, we know $p$ has a unique projection $\xi_A(p) = q$. By assumption $q \notin \ball{\epsilon}{x}$ and therefore $q \neq x$. As $q$ is the nearest neighbour of $p$ in $A$, we have  $r = \norm{p-q} < \norm{p-x}  < \epsilon < \reach_A$.
		
		Let $q(t) = (1-t)q + tp$ for $t \in [0,1]$ parametrise the line segment $\overline{qp}$. Since $q \in \ballc{tr}{q(t)} \cap A \subset \ballc{r}{q} \cap A  = q$ for all $t \in [0,1]$, the entire line segment $\overline{qp}$ lies in $\UP{A,q}$ Since .
		
		We consider the continuous function $c(t) = \norm{q(t) - x}$. Since  $g(0) > \epsilon$ and $g(1) < \epsilon$, by continuity there must be some $t_\ast \in (0,1)$ such that $c(t_\ast) = \norm{q(t_\ast) - x} = \epsilon$. Since $p$ is in the interior of  $\ball{\epsilon}{x}$ and $\ball{\epsilon}{x}$ is convex, $t_\ast$ must be unique. Thus for $y = q(t_\ast)$, we have $\overline{yp} \subset \ballc{\epsilon}{x}$.
		
		Since $\pi(y) = q \in A$ and $x \in A$, we can now apply \Cref{lem:angletonormal,lem:faketau0}, and deduce that
		\begin{equation*}
		\expval{\normvec{y-q},\normvec{x-q}} \leq \frac{\norm{x-q}}{2\reach_A}.
		\end{equation*}
		As such,
		\begin{align*}
		\epsilon^2 = \norm{y-x}^2 &= \norm{(y-q) + (x-q)}^2\\
		&=\norm{y-q}^2 + \norm{x-q}^2 - 2\expval{y-q, x-q}\\
		&\geq  \norm{y-q}^2 + \norm{x-q}^2 \qty(1- \frac{\norm{y-q}}{\reach_A})\\
		&\geq  \norm{y-q}^2 + \epsilon^2 \qty(1- \frac{\norm{y-q}}{\reach_A})\\
		\implies \norm{y-q} &\leq \frac{\epsilon^2}{\reach_A}.
		\end{align*}
		where in the fourth line we applied our assumption that $q \notin \ball{\epsilon}{x}$ and the fact that $\norm{y-q} = \norm{c(t_\ast) - q} < \reach_A$.
	\end{proof}
	
	Here is the promised proof of \Cref{prop:nonoise}.
	
	\begin{proof}[Proof of \Cref{prop:nonoise}]
		Consider any point $p \in \X^\epsilon$ and let $\xi_A(p)= q$. Since  $p \in \X^\epsilon$, the point $p$ is contained in at least one open Euclidean ball $\ball{\epsilon}{x}$ for some $x \in \X$. If $q \in \ball{\epsilon}{x}$ for one such $x$, then $\overline{pq}$ is also contained within $\ball{\epsilon}{x}$ as Euclidean balls are convex. Therefore, $\overline{pq}$ is contained in $\X^\epsilon$.
		
		Let us consider the case where $q$ is not contained in any of the Euclidean balls containing $p$. From \Cref{lem:nonoiseoffsetbound}, there exists some $y \in \overline{qp}$ such that $\overline{pq} \subset \ballc{\epsilon}{x}$, and $\norm{y-q} \leq \epsilon^2/\reach_A$. We can subdivide the line segment $\overline{pq}$ into two segments $\overline{qy}$ and $\overline{py}$, the latter being contained in $ \ball{\epsilon}{x}$. If there is some $x' \in \X$ such that the closed line segment $\overline{qy}$ is contained in $\ball{\epsilon}{x'}$, then the entirety of $\overline{pq}$ is contained in $\X^\epsilon$. This can be achieved if both $q$ and $y$ are contained in $\ball{\epsilon}{x'}$. By the $\delta$-density assumption, we can pick a point $x'\in \X$ such that $q \in \ball{\delta}{x'}$. If we assume $\delta < \epsilon$, then $q \in \ball{\epsilon}{x'}$.
		
		If $y$ is to be contained in $\ball{\epsilon}{x'}$, we require $\norm{x'-y} < \epsilon$. If we choose $\delta < \epsilon - \epsilon^2/\bar{\reach}$, then by the triangle inequality,
		\begin{align*}
		\norm{x'-y} &\leq \norm{x'-y} + \norm{q-y} \\
		&< \delta + \frac{\epsilon^2}{\reach_A}\\
		& < \epsilon - \frac{\epsilon^2}{\bar{\reach}}  +  \frac{\epsilon^2}{\reach_A}  \leq \epsilon.
		\end{align*}
		Therefore for choices of $\epsilon$ that satisfy $\delta < \epsilon - \epsilon^2/\bar{\reach}$, the line segment $\overline{pq}$ is contained in $\X^\epsilon$. Since this holds for any choice of $p$, by the \cref{lem:defretr}, this implies $\X^\epsilon$ deformation retracts to $A$.
		
		Finally, we have $\delta < \epsilon - \epsilon^2/\bar{\reach}$ if and only if
		\begin{equation*}
		\frac{\epsilon}{\bar{\reach}} \in \qty(\frac{1}{2}- \sqrt{\frac{1}{4}-\frac{\delta}{\bar{\reach}}}, \frac{1}{2} + \sqrt{\frac{1}{4}-\frac{\delta}{\bar{\reach}}}).
		\end{equation*}
	\end{proof}
	
	\section{Homological Inference of Index Pairs from Point Samples} 
	\label{sec:hominfidxpr}
	Having constructed an index pair $(\GM{N}, \GM{N}_-)$ of $S$ in \Cref{sec:constrindxpr}, we will show in this Section that the Conley index $\Con_\bullet(S) \cong \HG_\bullet(\GM{N}, \GM{N}_-)$ can be inferred from a sufficiently large finite point sample $\X \subset \GM{N}$ as described in Theorem (A) of the Introduction. We proceed by showing that $\GM{N}$ and $\GM{N}_-$ are both regular intersections, and thus have positive reach.

	\begin{lemma} \label{prop:regindexpair} If the Hessians of $f$ and $g$ are bounded, then $\GM{N}$ and $\GM{N}_-$ as constructed for $S \in \ccrit{f}$ as defined in \cref{eq:GMM} are regular intersections.
		\begin{proof}
			As the Hessians of $f$ and $g$ is bounded and $q$ has bounded second derivative, the Hessian of $h$ (\cref{eq:GMh}) is also bounded. Thus \cref{def:regcap2} is satisfied.
			
			Since $\crit{f} = \crit{h}$ (\Cref{H5}) and $S = \crit{f} \cap \GM{N} \subset \suplevelopen{h}{\gamma} \cap \sublevelopen{f}{\alpha}$ (\Cref{lem:IP1}),  $\nabla f$ and $\nabla h$ are non-zero on level sets $\level{f}{\alpha}$, $\level{h}{\beta}$, and $\level{h}{\gamma}$. Because we have assumed \Cref{G6}, the Jacobian of $(f,h): M \to \R^2$ is surjective on $\level{f}{\alpha} \cap \level{h}{\beta}$ and $\level{f}{\alpha} \cap \level{h}{\gamma}$ respectively. Since these are bounded and thus compact, the infimum of the second largest singular value of the Jacobian of $(f,h)$ on these sets is positive. We thus satisfy \Cref{def:regcap1}.
		\end{proof}
	\end{lemma}
	
	Given this characterisation of $(\GM{N}, \GM{N}_-)$ in terms of regular intersections, the following homology inference guarantee follows from \Cref{cor:regintdense}.
	
	\begin{restatable}{proposition}{indexpairsample}
		\label{prop:indexpairsample}
		Let $\GM{N}$ and $\GM{N_-}$ be as constructed in \cref{eq:GMM}; these are $(\mu, \Lambda)$-regular intersections for parameters $\mu > 0$ and $\Lambda > 0$ by \Cref{prop:regindexpair}. Fix $\delta \in (0, \frac{\rho_2}{4})$, where
		\begin{equation*}
		\frac{1}{\rho_2} = \frac{1}{\reach_M} + \sqrt{2}\frac{\mu}{\Lambda},
		\end{equation*} Let $\X \subset \GM{N}$ be a finite point sample that is $\delta$-dense in $\GM{N}$ so that $\X_- := \X \cap \GM{N}_-$ is $\delta$-dense in $\GM{N_-}$. 
		Then $\HG_\bullet\qty(\GM{N}, \GM{N}_-) \cong \HG_\bullet(\X^\epsilon, \X_-^\epsilon)$ for $\epsilon$ in the open interval
		\begin{equation}
		\frac{\epsilon}{\rho_2} \in \qty(\frac{1}{2}- \sqrt{\frac{1}{4}-\frac{\delta}{\rho_2}}, \frac{1}{2} + \sqrt{\frac{1}{4}-\frac{\delta}{\rho_2}}). \label{eq:indexpairsampleeps}
		\end{equation}
	\end{restatable}
	
	\begin{proof}
		A positive lower bound on the reaches of $\GM{N}$ and $\GM{N}_-$ is prescribed by \Cref{lem:regreachbound}. In particular, since each point in either $\GM{N}$ or $\GM{N}_-$ lies at the intersection of at most two level sets of $f$ and $h$, the reaches of $\GM{N}$ and $\GM{N}_-$ are bounded below by $\rho_2 > 0$ (where $\rho_2$ is given by \cref{eq:rhok} for the case $k = 2$). We can then apply \Cref{prop:nonoise} to recover the homotopy type of $\GM{N}$ and $\GM{N}_-$ from  $\X^\epsilon$ and $\X_-^\epsilon$. As we have isomorphisms $\HG_\bullet(\GM{N}) \xrightarrow{\cong}\HG_\bullet(\X^\epsilon)$ and  $\HG_\bullet(\GM{N}_-) \xrightarrow{\cong} \HG_\bullet(\X_-^\epsilon)$, applying the five lemma to the two long exact sequences of the pairs $(\X^\epsilon, \X_-^\epsilon)$ and $(\GM{N}, \GM{N}_-)$ gives the isomorphism $\HG_\bullet\qty(\GM{N}, \GM{N}_-) \cong \HG_\bullet(\X^\epsilon, \X_-^\epsilon)$.
	\end{proof}

	\subsection{Homology Inference via Uniform Sampling}
	
	We now consider random i.i.d. samples from a compactly supported measure on a Riemannian manifold $M$. A Riemannian manifold is endowed with a unique Riemannian measure $\mu_g$; for $p \in M$ and coordinate neighbourhood $(U,\set{x_i})$ of $p$ such that $\dd{_{p}x_i}$ are orthonormal in the cotangent space $T^\ast_pM$,
	\begin{equation}
	\mu_g(p) = \abs{\dd{_px_1} \wedge \cdots \wedge \dd{{}_px_m}}.
	\end{equation}
	
	The authors of \cite{niyogi_finding_2008} studied how many random i.i.d. draws from a probability measure $\nu$ would suffice to obtain an $\epsilon$-dense point sample of some measurable subset of $M$. We consider $\nu$ that is continuous with respect to $\mu_g$: that is, for any measurable subset $A \subset M$, $\mu_g(A) = 0$ implies $\nu(A) = 0$. Here we will follow the argument from \cite[Proposition 7.2]{niyogi_finding_2008}, summarised below.
	
	\begin{restatable} {lemma}{nswN}
		Let $A$ be a compact, measurable subset of a properly embedded Riemannian submanifold of $\R^d$. Suppose $\nu$ is a probability measure supported on $A$ that is continuous with respect to $\mu_g$, and let $\X$ be a finite set of i.i.d. draws according to $\nu$. If
		\begin{equation}
		\# \X \geq \frac{1}{K\qty(\frac{\epsilon}{2})} \qty(\log(\frac{\nu(A)}{K\qty(\frac{\epsilon}{4})}) + \log(\frac{1}{\kappa})). \label{eq:nswpointbound}
		\end{equation}
		where
		\begin{equation}
		K(r) = \inf_{p \in A} \nu\left(A \cap \ball{r}{p}\right), \label{eq:Kr}
		\end{equation}
		then $\X$ is $\epsilon$-dense in $A$ with probability $> 1-\kappa$. \label{lem:NSWdensity}
	\end{restatable}
	
	The key parameter that controls the bound on the right hand side of \cref{eq:nswpointbound} is $K(r)$; in particular, we require $K(r) >0$ for $r>0$ so that the bound is finite. That is the case if $A$ is a compact,  regular closed subset of $M$.
	
	\begin{lemma} \label{lem:Kpos}
		Let $A \subset M$ be a compact regular closed subset. For any $r > 0$, and measure $\nu$ supported on $A$ that is continuous with respect to $\mu_g$ on $M$,
		\begin{equation*}
		K(r) = \inf_{p \in A} \nu(\ball{r}{p} \cap A) > 0.
		\end{equation*}
		\begin{proof}
			We first show that $\nu(\ball{r}{p} \cap A) > 0$ for $p \in E$. Since $\nu$ is  supported on $A$, and continuous with respect to $\mu_g$, it suffices to show that $\ball{r}{p} \cap A$ contains some open subset of $M$ for all $p \in A$. Since $A$ is a regular closed subset of $M$, $p \in \closure{ \interior{A}} = A$, and thus we also have $\ball{r}{p} \cap \interior{A} \neq \emptyset$. Since $\ball{r}{p} \cap A$ contains a non-empty open set $\ball{r}{p} \cap \interior{A}$ for any $p \in A$, we therefore surmise that $\nu(\ball{r}{p} \cap A) > 0$. As $\nu(\ball{r}{p} \cap A)$ is continuous with respect to $p$ and positive (\Cref{lem:measurecont}), and $A$ is compact, its infimum $K(r)$ over $p \in A$ is positive.
		\end{proof}
	\end{lemma}
	
	As regular intersections are regular closed sets (\Cref{lem:regcapisreg}), \Cref{lem:Kpos} implies that a finite sampling bound on the type in \cref{eq:nswpointbound} can be obtained if $K(r)$ is bounded below away from zero. We focus on the case where the regular intersection can be written as an intersection of two function sublevel sets for simplicity of analysis. Furthermore, our index pairs constructed above
	\begin{equation*}
	\GM{N} = \suplevel{h}{\beta} \cap \sublevel{f}{\alpha} \qand \GM{N}_- = \sublevel{f}{\alpha} \cap \intlevel{h}{\beta}{\gamma}, \tag{\cref{eq:GMM}}
	\end{equation*}
	fall under this category: $\GM{N}_-$ is the intersection of $\sublevel{f}{\alpha}$ with the sublevel set $\sublevel{\tilde{h}}{0}$ of $\tilde{h} = (h-\beta)(h-\gamma)$. The remainder of this section is devoted to proving the following result.
	
	\begin{restatable}{theorem}{Ebound}
		\label{thm:ebound}
		Let $E = \sublevel{f_1}{0} \cap \sublevel{f_2}{0}$ be a regular intersection on $M$ and $ \level{f_1}{0} \cap \level{f_2}{0} \neq \emptyset$. Let $\X$ be a finite set of i.i.d. draws according to the Riemannian density $\mu_g$ on $M$, restricted to $E$. Then for $\epsilon < \rho_2$, if
		\begin{equation}
		\# \X \geq \frac{1}{K\qty(\frac{\epsilon}{2})} \qty(\log(\frac{\nu(E)}{K\qty(\frac{\epsilon}{4})}) + \log(\frac{1}{\kappa})).
		\end{equation}
		then $\X$ is $\epsilon$-dense in $A$ with probability $> 1-\kappa$, where
		\begin{equation}
		K(r) = \inf_{p \in A} \mu_g(A \cap \ball{r}{p}) \geq \vhstwo{\frac{\overline{\eta}_E(r)}{2} \cos(\theta\qty(\frac{\overline{\eta}_E(r)}{2}))}{\rho}{\phi_{12}} > 0,
		\end{equation}
		and  $\phi_{12}$, $\vhstwo{\cdot}{\cdot}{\cdot}$, and $\overline{\eta}_E$ are as defined in \Cref{lem:regintproj2}, \Cref{nota:balls}, and \Cref{prop:bottlethickbound} respectively.
	\end{restatable}
	This bound captures two aspects of how the volume of a Euclidean ball intersected with $M$ can be diminished when we restrict to $E$. One aspect is the local geometry near the `cusps' of the intersection $\level{f_1}{0} \cap \level{f_2}{0}$, and this is parametrised by an angle parameter $\phi_{12}$. This angle, which we formally define in \Cref{lem:regintproj2}, is the smallest angle subtended by $\nabla f_1$ and $\nabla f_2$ in $T_pM$ for $p \in \level{f_1}{0} \cap \level{f_2}{0}$, and $\vhstwo{\cdot}{\cdot}{\phi_{12}}$ is a strictly decreasing function of $\phi_{12}$ which is non-zero for  $\phi_{12} < \pi$. The closer $\nabla f_1$ and $\nabla f_2$ are to aligning oppositely, the thinner the intersection is near the cusps, and the smaller the volume.
	
	The other pertinent aspect of the geometry of $E$ is the thickness of $E$ away from the cusps. This is parametrised by $\overline{\eta}_E(r)$, which takes into account the curvature effects of the individual level sets $\level{f_i}{0} \cap \partial E$  in the boundary, as well as the thickness at which the two individual level sets approach each other. The interaction of the two level sets is parametrised by what we call a bottleneck thickness, which we define in \Cref{def:bottleneckthickness}, and illustrate in \Cref{fig:bottle}. We provide a lower bound on $\overline{\eta}_E(r)$ in terms of properties of the constituent functions $f_1$ and $f_2$ in \Cref{prop:bottlethickbound}.
	
	\begin{notation} \label{nota:balls}
		Our volume bounds will be phrased in terms of the volumes of the following geometric quantities and objects:
		\begin{enumerate}
			\item For a manifold with reach $\reach_M$, let $\theta(r) = \arcsin(\frac{r}{2\reach_M})$;
			\item $\mathscr{B}_r$ denotes a ball of radius $r$ in $T_pM$, centered at the origin of $T_pM$ which is identified with $p$ in the ambient Euclidean space. We denote the volume of $\mathscr{B}_{r}$ as
			\begin{align}
			\vball{r} = \vol\qty(\mathscr{B}_{r}).
			\end{align}
			\item For $v \in T_pM$, we let $\mathfrak{B}(v)$ denote the ball in $T_pM$ centred at $v$ with radius $\norm{v}$ (so that $v$ is a radial vector of the ball). We let
			\begin{align}
			\vhs{r,s} := \vol\qty( \mathfrak{B}(su) \cap \mathscr{B}_{r})
			\end{align}
			be the volume of the hyperspherical cap that arises from the intersection on the right and side.
			for any unit vector $u \in \S_pM$. If $t$ is another unit vector in $\S_pM$, and $\anglegeo{u}{t} = \varphi > 0$, then we let
			\begin{align}
			\vhstwo{r}{s}{\varphi} := \vol\qty(\mathfrak{B}(su) \cap \mathfrak{B}(st) \cap \mathscr{B}_{r}).
			\end{align}
			We give an illustration of $\mathfrak{B}(su) \cap \mathfrak{B}(st) \cap \mathscr{B}_{r}$ in the two dimensional case in \Cref{fig:vhs2}.
		\end{enumerate}
	\end{notation}
	
	\begin{figure}
		\centering
		\includegraphics[width =  0.45\textwidth]{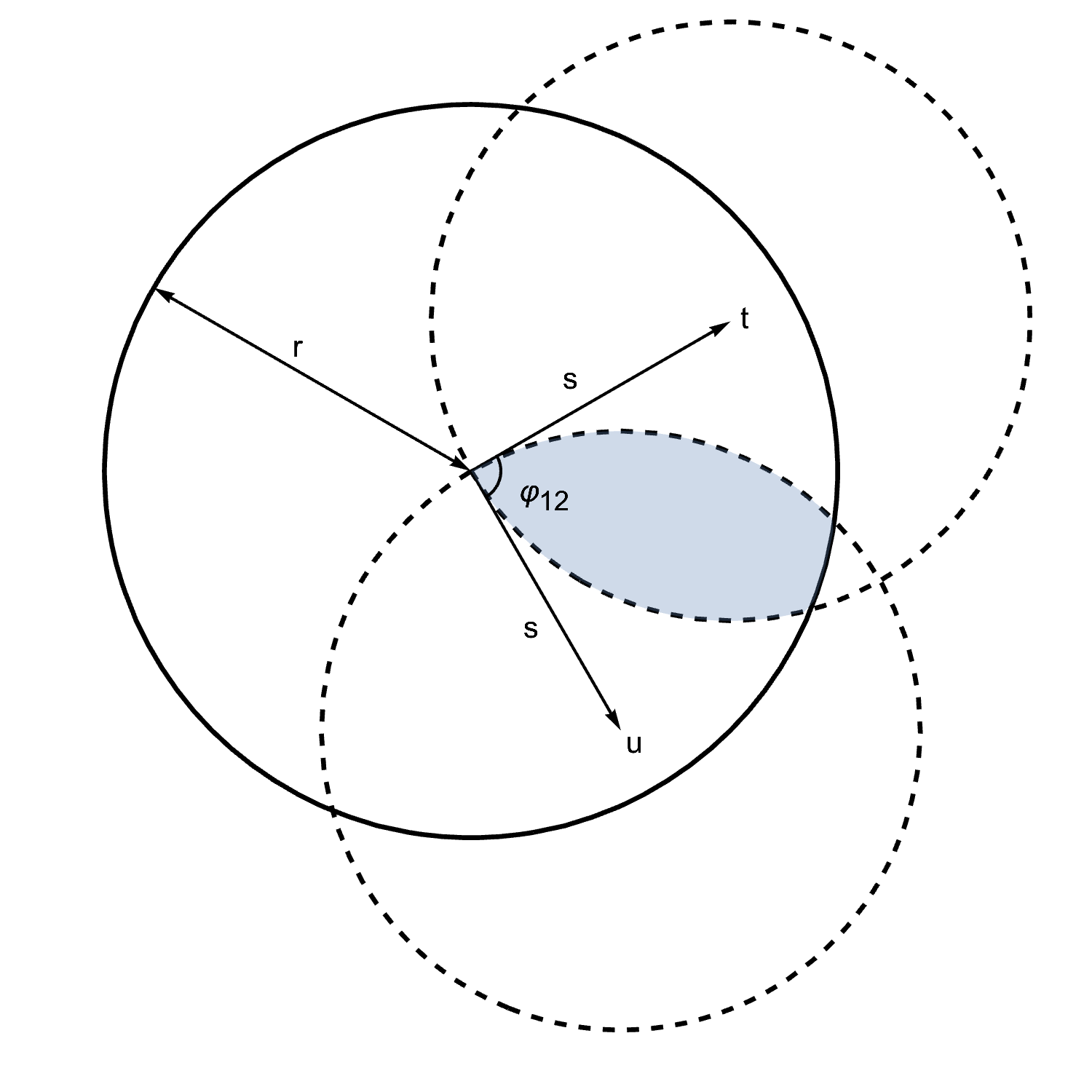}
		\caption{An illustration of the volume $\mathfrak{B}(su) \cap \mathfrak{B}(st) \cap \mathscr{B}_{r}$ (as defined in \Cref{nota:balls}) in the plane containing vectors $s$ and $t$ in $T_pM$. The shaded region represents the intersection of the three balls, where the ball with solid boundary represents $\mathscr{B}_r$, and the the balls with dashed boundaries represent $\mathfrak{B}(su)$ and $\mathfrak{B}(st)$ respectively. }
		\label{fig:vhs2}
	\end{figure}

	\begin{figure}
		\centering
		\includegraphics[width =  0.45\textwidth]{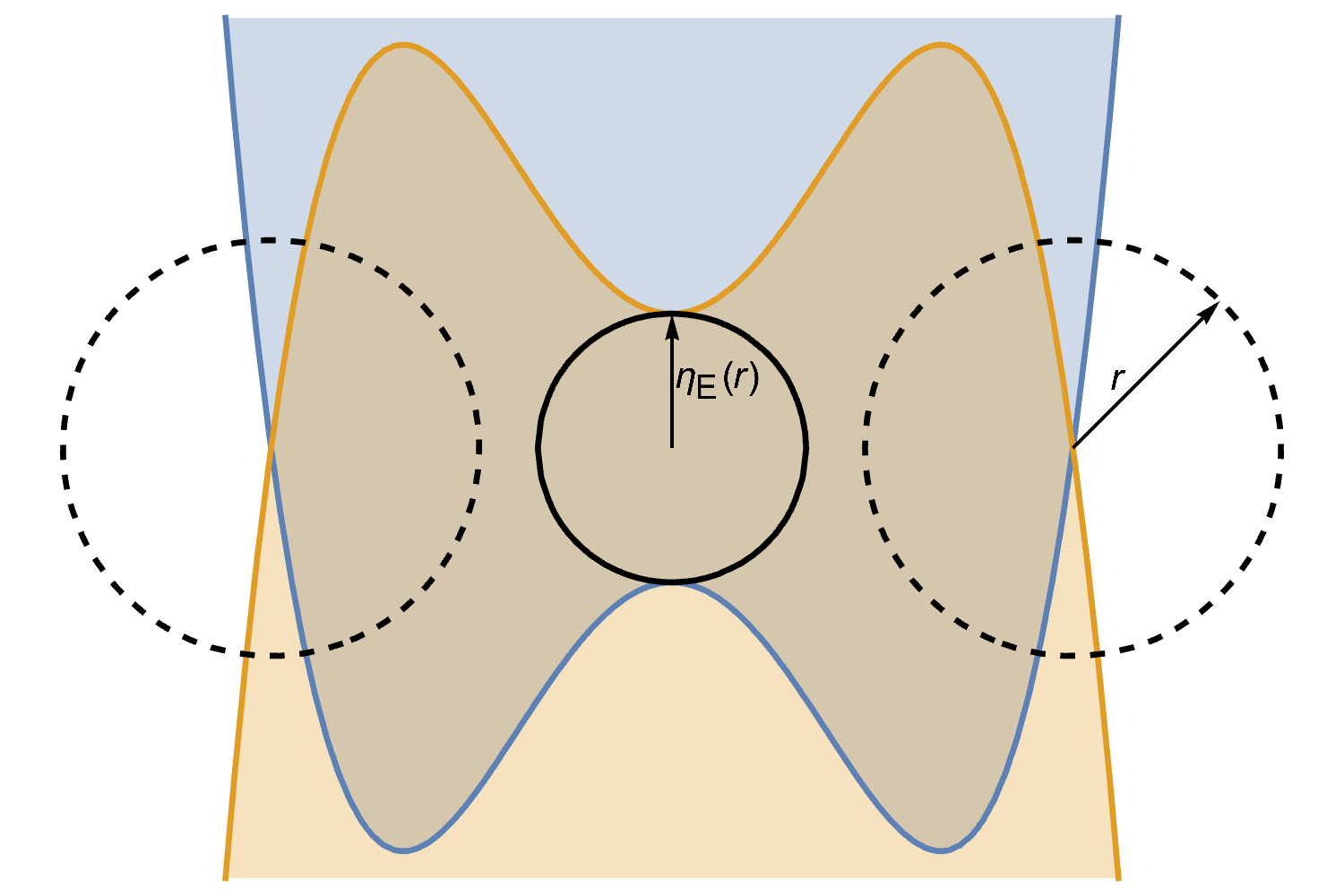}
		\caption{An illustration of the bottleneck distance $\eta_E(r)$ of a regualar intersection of two sublevel sets. }
		\label{fig:bottle}
	\end{figure}
	
	\subsection{Review: the Case of Manifolds and Regular Domains}
	In \cite{niyogi_finding_2008}, they bounded the volumes and probability measure of $A \cap \ball{r}{p}$  (where in their case $A = M$) by considering the volume of its image under the orthogonal projection $\zeta_p: M \to \tp{p}{M}$ onto the tangent plane $\tp{p}{M}$. We proceed with the same argument and begin by stating the following two lemmas that \cite{niyogi_finding_2008} implicitly relies on. These lemmas in turn rely on the result in \Cref{prop:ballembed} which gives a lower bound on the maximum radius $r$ such that $\zeta_p$ is a diffeomorphism onto its image when restricted to $\ball{r}{p} \cap M$.
	
	We first recall if we have a smooth map $F: M \to N$ between two Riemannian manifolds of the same dimensions, then we can pull back a density $\nu$ on $N$ to obtain a density $F^\ast\nu$ on $M$. Locally, if $\nu = u\abs{\dd{y_1} \wedge \cdots \wedge \dd{y_m}}$ on some local coordinate patch $(V, (y_i))$ of $F(p)$, then we can locally express the pullback $F^\ast\nu$ as
	\begin{align}
	F^\ast\nu(p) = F^\ast(u\abs{\dd{y_1} \wedge \cdots \wedge \dd{y_m}}) (p) = \abs{\det (\dd{{}_pF})} \cdot u(F(p)) \cdot \abs{\dd{x_1} \wedge \cdots \wedge \dd{x_m}}.
	\end{align}
	where $(U, (x_i))$ is a local coordinate patch of $p$ in $M$, and $\dd{F}$ is the Jacobian matrix of $F$ with respect to coordinates $(x_i)$ and $(y_i)$ on $U$ and $N$ respectively. The pullback satisfies an important property: if $F$ is a diffeomorphism, then one can show that
	\begin{align}
	\int_M F^\ast\nu = \int_N \nu.
	\end{align}
	
	\begin{lemma}\label{lem:projvollb} Suppose $U \subseteq \ball{r}{p} \cap M$ where $r \leq \reach_M$. Then
		\begin{align}
		\vol\qty({U}) \geq \vol\qty({\zeta_p(U)}).
		\end{align}
		where $\vol(U) = \mu_g(U)$, and  $\vol({\zeta_p(U)})$ is the $m$-dimensional volume of $\zeta_p(U)$ in the tangent plane.
	\end{lemma}
	
	\begin{proof} Let $\nu$ be the Riemannian density of the tangent plane. Recall $\zeta_p$ restricted to $U$ is a diffeomorphism onto its image and let $F = \inv{\zeta_p}$ be its inverse. Then
		\begin{align*}
		\vol({U}) = \int_{U} \mu_g  = \int_{\zeta_p(U)} F^\ast\mu_g .
		\end{align*}
		Let $y_1, \ldots, y_m$ be a set of coordinates on the tangent plane that are orthonormal with respect to the ambient Euclidean metric restricted to the tangent plane, and consider the density at $z = \zeta_p(q) \in \zeta_p(U)$. If $x_1, \ldots, x_m$ is an orthonormal set of coordinates at $q \in T_qM$, then we can explicitly write the density $F^\ast\mu_g$ as
		\begin{align*}
		F^\ast\mu_g (z) &= F^\ast\abs{\dd{x_1} \wedge \cdots \wedge \dd{x_m}}(z) \\
		&= \abs{\det \dd{F}}\abs{\dd{y_1} \wedge \cdots \wedge \dd{y_m}}\\
		&= \frac{1}{\abs{\det \dd{\zeta_p}}}\abs{\dd{y_1} \wedge \cdots \wedge \dd{y_m}} \\
		& \geq \abs{\dd{y_1} \wedge \cdots \wedge \dd{y_m}}.
		\end{align*}
		where the inequality in the final line is due to $\abs{\det \dd{\zeta_p}} \leq 1$ (\Cref{lem:projdet}). Since $\abs{\dd{y_1} \wedge \cdots \wedge \dd{y_m}}$ is the Riemannian density on the tangent plane at $z = \zeta_p(q)$,
		\begin{align*}
		\vol\qty(U) = \int_{\zeta_p(U)} F^\ast\mu_g  \geq \vol\qty(\zeta_p(U)).
		\end{align*}
	\end{proof}
	
	We now derive a lower bound on $\vol\qty(\ball{\epsilon}{p})$ by describing the image of such balls in $T_pM$ under the projection $\zeta_p$. We first recall such descriptions for manifolds with boundary, where \cite{wang_topological_2020} differentiates into two cases where $p \in \partial M$ or $p \in \interior{M}$.
	
	\begin{lemma}[{\cite[Lemma 5.3]{niyogi_finding_2008}} and {\cite[Lemma 4.7]{wang_topological_2020}}] Let $M$ be a properly embedded manifolds with positive reach. Recall \Cref{nota:balls}, we have for any $p \in M$ and $r< \reach_M$,
		\begin{align}
		\zeta_p\qty(\ball{r}{p} \cap M) \supset \mathscr{B}_{r\cos(\theta(r))}
		\end{align}
		Consequently, $\vol\qty(\ball{r}{p} \cap M) \geq \vball{r\cos(\theta(r))}$.
		\label{lem:nswlem503}
	\end{lemma}

	\begin{lemma}[{\cite[Lemma 4.6]{wang_topological_2020}}] Let $M \subset \R^d$ be a properly embedded submanifold with boundary. For $p \in \partial M$, let $n \in T_pM \cap N_p \partial M$ be the inward pointing unit normal at $p$. If $\zeta_p$ is a diffeomorphism when restricted to $\ball{\delta}{p} \cap M$, where $\delta < \min(\reach_M, \reach_{\partial M})$, then for $\epsilon \in (0,\delta)$,
		\begin{equation*}
		\zeta_p(\ball{\epsilon}{p} \cap M) \supset \mathscr{B}_{\epsilon\cos(\theta(\epsilon))} \cap \mathfrak{B}(\epsilon n)
		\end{equation*}
		where we recall notaion from \Cref{nota:balls}. Consequently, $\vol\qty(\ball{r}{p} \cap M) \geq \vhs{\epsilon \cos(\theta(\epsilon))}{\epsilon}$.
		
		\label{lem:wang}
	\end{lemma}
	
	Applying \Cref{prop:ballembed} to deduce the minimal radius $\delta$ in  \Cref{lem:wang}, we obtain the following result for regular domains.
	
	\begin{lemma} \label{lem:sublevelproj} Consider $\sublevel{f}{0}$, where $f: M \to \R$ is a smooth function on a submanifold $M$ of $\R^d$ with positive reach $\reach_M$ and $f$ satisfies \Cref{def:regcap1,def:regcap2}. Suppose for any $p \in \level{f}{0}$, we have $\norm{\nabla_p f} \geq \mu > 0$. Let
		\begin{equation}
		\frac{1}{\rho} = \frac{1}{\reach_M} + \frac{\Lambda}{\mu}.
		\end{equation}
		Recall \Cref{nota:balls}. If $\epsilon < \rho$, then
		\begin{enumerate}[start = 1, label={ ({\roman*})}]
			\item For $p \in \level{f}{0}$, and $n = -\normvec{\nabla_p f}$,
			\begin{equation*}
			\zeta_p(\ball{\epsilon}{p} \cap \sublevel{f}{0}) \supset  \mathscr{B}_{\epsilon\cos(\theta(\epsilon))} \cap \mathfrak{B}(\rho n), 
			\end{equation*}
			\item  Consequently, for \emph{any} $q \in \sublevel{f}{0}$,
			\begin{equation}
			\vol\qty(\ball{\epsilon}{q} \cap E) \geq \vhs{\frac{\epsilon}{2} \cos(\theta\qty(\frac{\epsilon}{2}))}{\rho}
			\end{equation}
		\end{enumerate}
		is a function of $\epsilon, \rho, \reach_M$, and the dimension of the manifold $m$.

		\begin{proof}\phantom{0}
			\begin{enumerate}[start = 1, label={ ({\roman*})}]
				\item Because $\rho < \reach_M$, \Cref{prop:ballembed} implies $\zeta_p$ is a diffeomorphism on $\ball{\rho}{p} \cap M$; thus, combined with the fact that $\rho < \min(\reach_{\level{f_i}{0}},\reach_{\sublevel{f_i}{0}})$ (\Cref{lem:regreachbound}), we can apply \Cref{lem:wang} and deduce the inclusion as stated.
				\item The proof adopts the proof of {\cite[Lemma 4.6]{wang_topological_2020}} for regular domains being a special case of submanifolds with boundary, which breaks down the analysis into two cases, whether $d(p, \partial E) > \frac{\epsilon}{2}$ or otherwise. In the first case, we can trivially bound $\vol\qty(\ball{\epsilon}{p} \cap E)$ from below by $\vol\qty(\ball{\epsilon/2}{p} \cap E) $. Since $\ball{\epsilon/2}{p} \cap  \partial E = \emptyset$, we have $\ball{\epsilon/2}{p} \cap   E =\ball{\epsilon/2}{p} \cap  M$. The volume of the latter is bounded below by $\vball{\frac{\epsilon}{2} \cos(\frac{\epsilon}{2})}$ (by \Cref{lem:nswlem503}), which is in turn bounded below by $\vhs{\frac{\epsilon}{2} \cos(\theta\qty(\frac{\epsilon}{2}))}{\rho}$ by definition (see \Cref{nota:balls}).
				
				For the case where $d(p, \partial E) \leq \frac{\epsilon}{2}$, choose a nearest neighbor $x$ of in $\partial E$. As $\ball{\epsilon/2}{x} \subset \ball{\epsilon}{p}$, we can bound the volume of the former by the latter. The latter's volume is bounded by that of its projection (\Cref{lem:projvollb}), and we arrive at the stated bound.
				
			\end{enumerate}
		\end{proof}
	\end{lemma}

	\subsection{Regular Intersections of Two Functions}
	We now consider a regular intersections of two function $E = \sublevel{f_1}{0} \cap \sublevel{f_2}{0}$.
	We bound the volume of $\ball{\epsilon}{p} \cap E$ from below by dividing the set of points $p \in E$ into two cases. Let $E_r$ denote the set of points
	\begin{equation}
	E_r = \sett{x \in E}{d\left(x, \level{f_1}{0} \cap \level{f_2}{0}\right) \geq r}.
	\end{equation}
	Then either:
	\begin{enumerate}[start = 1, label={ ({\bfseries Case \arabic*})}]
		\item  \label{onedimI} $p \in E_\epsilon$, i.e., $\ball{\epsilon}{p}$ may intersect both $\level{f_1}{0}$ and $\level{f_1}{0}$, but not $\level{f_1}{0} \cap \level{f_2}{0}$; else
		\item  \label{twodimI} $p \notin E_\epsilon$, i.e., $\ball{\epsilon}{p}$ intersects $\level{f_1}{0} \cap \level{f_2}{0}$.
	\end{enumerate}
	In \Cref{onedimI}, we require an additional lengthscale to control the geometry of $E$ in the $\epsilon$-ball about $p$.
	
	\begin{definition}\label{def:bottleneckthickness} Let $E$ be a regular intersection where $E = \sublevel{f_1}{0} \cap \sublevel{f_2}{0}$. The $r$-{\bf bottleneck thickness} $\eta_E(r)$ of $E$ is
		\begin{equation}
		\eta_E(r) = \sup \sett{\epsilon \in (0,  r) }{\forall x \in E_r\qc \ball{\epsilon}{x} \cap \level{f_i}{0}\neq \emptyset \text{ for at most one } i \in \qty{1,2}}.
		\end{equation}
	\end{definition}
	Note that in the case where $\level{f_1}{0} \cap \level{f_2}{0} = \emptyset$ (i.e. $E = E_r$ for all $r$), we can write $\partial E =  \level{f_1}{0} \sqcup \level{f_2}{0}$. The bottleneck thickness $\eta_E$ is then the largest radius for which we can thicken $\partial E$ such that the thickenings of $\level{f_1}{0}$ and $\level{f_2}{0}$ do not intersect. Thus $\eta_E$ is a homological critical value in the thickening of $\partial E$, implying the bottleneck thickness $\eta_E$ is an upper bound on the reach of the manifold $\partial E$ in $\R^d$. Furtheremore, writing $\partial E = \level{F}{0}$ for $F = f_1f_2$, one can check that $\dd{F}$ is nowhere zero on $\partial E$ and we can derive an explicit bound on the reach of $\partial E$ using \Cref{lem:fsingle}, thus in turn providing a lower bound on $\eta_E$.
	
	In the more general case, we can also interpret $\eta_E$ as thus. For any point $x$ on $E$ that are at least $r$ away from  $\level{f_1}{0} \cap \level{f_2}{0}$ on $E$, the bottleneck thickness $\eta_E(r)$ prescribes the largest radius such that for  $t < \eta_E(r)$,
	\begin{equation*}
	\ball{t}{x} \cap E = \ball{t}{x} \cap \sublevel{f_i}{0}
	\end{equation*}
	for $i$ either $1$ or $2$. Thus, for such points, we can apply results such as \Cref{lem:sublevelproj} to bound the volume of $\ball{r}{x} \cap M$. We now derive a positive lower bound on $\eta_E(r)$.
	
	\begin{proposition} \label{prop:bottlethickbound} Let $E$ be a regular intersection where $E = \sublevel{f_1}{0} \cap \sublevel{f_2}{0}$. Let $\reach_i$ denote the reach of the submanifold $\level{f_i}{0}$ for $i \in \qty{1,2}$. Let $F = f_1f_2: M \to \R$. For $0 < r/2 < \min\qty(\reach_1, \reach_2)$, let $\mu_{F}(r) = \inf_{x \in E_{r/2} \cap \level{F}{0}} \norm{\nabla F}$, and $\Lambda_{F}$ be the supremum of the norm of the Hessian of $F$ on $M$. Let
		\begin{equation}
		\frac{1}{\rho_F(r)} = \frac{1}{\reach_M} + \frac{\Lambda_{F}}{\mu_{F}(r)}.
		\end{equation}
		Then  $\eta_E(r) \geq \overline{\eta}_E(r)$ where
		\begin{equation}
		\overline{\eta}_E(r) = \min\qty(\frac{r}{2}, \rho_F(r)) > 0.
		\end{equation}
	\end{proposition}
	
	\begin{proof}
		We show that $\rho_F(r) > 0$ by showing $\mu_{F}(r) > 0$. For $x \in E_{r/2} \cap \level{F}{0}$, since $E_{r/2}$ excludes  $\level{f_1}{0} \cap \level{f_2}{0}$, we observe that $f_1(x)$ and $f_2(x)$ cannot both be zero at $x$. Moreover, since $E$ is a regular intersection, by \Cref{def:regcap}  we have $\dd{f_i} \neq 0$. Thus, for such a point $x$, either
		\begin{equation*}
		\dd{F}(x) = f_1\dd{f_2} \neq 0 \qq{or} \dd{F}(x) = f_2\dd{f_1} \neq 0.
		\end{equation*}
		Furthermore, one can check that $E_{r/2} \cap \level{F}{0}$ is compact and thus the quantity \[\mu_{F}(r) = \inf_{x \in E_{r/2} \cap \level{F}{0}} \norm{\nabla F}\] is positive.
		
		Consider then the Euclidean ball $\ball{\epsilon}{p}$ where $p \in E_r$ and $ \epsilon < \frac{r}{2}$. Since we have restricted $p \in E_r$, the ball $\ball{\epsilon}{p}$ cannot intersect $\level{f_1}{0} \cap \level{f_2}{0}$. Assume
		\begin{align*}
		\ball{\epsilon}{p} \cap \level{f_1}{0} \neq \emptyset, \qand \ball{\epsilon}{p} \cap \level{f_2}{0} \neq \emptyset.
		\end{align*}
		Since $\epsilon <  r/2 < \reach_i$ by assumption, $p$ has unique projections onto $\level{f_1}{0}$ and $\level{f_1}{0}$ respectively (\Cref{thm:fedtube}); let these projections be $p_i$ and note that $\norm{p - p_i} < r/2$. Note that $f_1(p_1) = f_2(p_2) = 0$, but $f_1(p_2), f_2(p_1) \neq 0$. As $p \in E_r$, and $\norm{p - p_i} < r/2$, we see that $p_i \in E_{r/2}$.
		
		Since $\ball{\epsilon}{p}$ cannot intersect $\level{f_1}{0} \cap \level{f_2}{0}$, we can suppose without loss of generality that $p \neq p_1$ and $\norm{p-p_1} \geq \norm{p-p_2}$. Since $p_1$ is a nearest neighbour of $p$ in $\level{f_1}{0}$, \Cref{lem:fsingle1} implies
		\begin{equation*}
		p - p_1 = n + \lambda \nabla f_1(p_1) = n + \frac{\lambda}{f_2(p_1)}\nabla F(p_1)
		\end{equation*}
		for some $n \in N_{p_1}M$. Because $p_1 \in \ball{\epsilon}{p}$ and $\ball{\epsilon}{p} \cap \level{f_1}{0} \cap \level{f_2}{0} = \emptyset$, we observe that $f_2(p_1) \neq 0$, and $\nabla F(p_1) =  f_2(p_1) \nabla f_1(p_1)$. Thus, we can write
		\begin{equation*}
		p - p_1 = n + \frac{\lambda}{f_2(p_1)}\nabla F(p_1).
		\end{equation*}
		In other words, the unit vector of $p-p_1$ lies in the normal space of $\level{F}{0}$ at $p$. If $\epsilon \leq \rho_F(r)$, then we have $\norm{p-p_1} < \rho_F(r)$. \Cref{lem:fsingle2} then implies $p_1$ is the \emph{unique} nearest neighbour of $p$ in $\level{F}{0}$. However, this contradicts our assumption that $\norm{p-p_1} \geq \norm{p-p_2}$, as $p_2 \in \level{F}{0}$ too. Therefore, $\epsilon  > \rho_F(r)$.
		
		Put in other words, either $\eta_E(r) \geq \frac{r}{2}$; else, $\eta_E(r) \geq \rho_F(r)$. We conclude that
		\begin{equation*}
		\eta_E(r) \geq \min\qty(\frac{r}{2}, \rho_F(r)).
		\end{equation*}
	\end{proof}
	
	Having obtained a lower bound for $\eta_E(r)$, we can bound the volume $\ball{r}{p} \cap E$ for $p \in E_r$.
	
	\begin{lemma} \label{lem:regintproj1} Consider $E = \sublevel{f_1}{0} \cap \sublevel{f_2}{0}$. Suppose $f_1$ and $f_2$ satisfy the conditions placed on $f$ in \Cref{lem:sublevelproj}, and $\rho$ be as defined in \Cref{lem:sublevelproj}. Let $\overline{\eta}_E(\epsilon)$ is as defined in \Cref{prop:bottlethickbound}. Then for $p \in E_\epsilon$,
		\begin{equation*}
		\vol\qty(\ball{\epsilon}{p} \cap E) \geq \vhs{\frac{\overline{\eta}_E(\epsilon)}{2} \cos\qty( \theta\qty(\frac{\overline{\eta}_E(\epsilon)}{2}))}{\rho}.
		\end{equation*}

		\begin{proof}
			By definition, $\overline{\eta}_E(\epsilon) \leq \epsilon$ (\Cref{prop:bottlethickbound}). As $\epsilon < \rho$, and $\rho$ is a lower bound on the reaches $\reach_i$ of $\level{f_i}{0}$ by \Cref{lem:regreachbound}), we have $\overline{\eta}_E(\epsilon) < \reach_i$. And the conditions of \Cref{prop:bottlethickbound} are satisfied so that $\overline{\eta}_E(\epsilon)$ is a positive lower bound on $\eta_E(\epsilon)$. Thus, for $p \in E_\epsilon$, the Euclidean ball $\ball{\overline{\eta}_E(\epsilon)}{p}$ can only intersect at most one level set $\level{f_i}{0}$. In other words, we can write without loss of generality that
			\begin{equation*}
			\ball{{\epsilon}}{p} \cap E \supset \ball{\overline{\eta}_E(\epsilon)}{p} \cap E = \ball{\overline{\eta}_E(\epsilon)}{p} \cap \sublevel{f_1}{0}
			\end{equation*}
			Since $f_1$ satisfies the conditions placed on $f$ in \Cref{lem:sublevelproj}, we apply the volume bound in  \Cref{lem:sublevelproj} to deduce
			\begin{equation*}
			\vol\qty(\ball{{\epsilon}}{p} \cap E ) \geq  \vol \qty(\ball{\overline{\eta}_E(\epsilon)}{p} \cap \sublevel{f_1}{0}) \geq  \vhs{\frac{\overline{\eta}_E(\epsilon)}{2} \cos\qty( \theta\qty(\frac{\overline{\eta}_E(\epsilon)}{2}))}{\rho}.
			\end{equation*}
		\end{proof}
	\end{lemma}
	
	\begin{lemma} \label{lem:regintproj2}
		Consider $E = \sublevel{f_1}{0} \cap \sublevel{f_2}{0}$, and suppose $f_1$ and $f_2$ satisfy the conditions placed on $f$ in \Cref{lem:sublevelproj}, and $\rho$ be as defined in \Cref{lem:sublevelproj}. Then, if $\epsilon < \rho$, and $p \in  \level{f_1}{0} \cap \level{f_2}{0}$, then
		\begin{equation*}
		\vol\qty(\ball{\epsilon}{p} \cap E) \geq \vol\qty(\ball{\epsilon\cos(\theta(\epsilon))}{p} \cap \mathfrak{B}(\rho n_1) \cap \mathfrak{B}(\rho n_1)) =: \vhstwo{\epsilon\cos(\theta(\epsilon))}{\rho}{\phi_{12}} > 0
		\end{equation*}
		where $\cos(\phi_{12}) = \inf_{x \in \level{f_1}{0} \cap \level{f_2}{0}} \abs{\expval{n_1(x), n_2(x)}} > -1$ for $n_i(x) = \normvec{\nabla_x f_i}$.
	\end{lemma}
	
	\begin{proof}
		Given $f_i$ satisfy the conditions placed on $f$ in \Cref{lem:sublevelproj}, and $\rho$ be as defined in \Cref{lem:sublevelproj},
		\begin{equation*}
		\zeta_p\left(\ball{\epsilon}{p} \cap \sublevel{f_i}{0}\right) \supset  \mathscr{B}_{\epsilon\cos \theta(\epsilon)} \cap  \mathfrak{B}(\rho n_i) 
		\end{equation*}
		Because $\epsilon < \rho$, the projection $\zeta_p$ is a diffeomorphism on $\ball{\epsilon}{p} \cap M$, and thus
		\begin{align*}
		\zeta_p\left(\ball{\epsilon}{p} \cap E\right) &=  \zeta_p\left(\ball{\epsilon}{p} \cap \sublevel{f_1}{0}\right)  \cap \zeta_p\left(\ball{\epsilon}{p} \cap \sublevel{f_2}{0}\right)  \\
		&\supset \mathscr{B}_{\epsilon\cos \theta(r)} \cap  \mathfrak{B}(\rho n_1)  \cap  \mathfrak{B}(\rho n_2) .
		\end{align*}
		As $E$ is a regular intersection, $n_1$ and $n_2$ are linearly independent (\Cref{def:regcap1}); and $\varphi_{12} > 0$ as the supremum is taken over a compact set by assumption that $E$ is a regular intersection (\Cref{def:regcap}).
		We now show that this implies $\mathfrak{B}(\rho n_1)  \cap  \mathfrak{B}(\rho n_2)  \neq \emptyset$. We make two observations, first, that $\mathfrak{B}(\rho n_1)$ and  $\mathfrak{B}(\rho n_2) $ are tangent to the origin; and second, the centres  $\rho n_1$ and $\rho n_2$ of  $\mathfrak{B}(\rho n_1) $ and  $\mathfrak{B}(\rho n_2) $ respectively are not collinear with the origin as $n_1$ and $n_2$ are linearly independent. Thus $\mathfrak{B}(\rho n_1)  \cap  \mathfrak{B}(\rho n_2)  \neq \emptyset$.
		
		Since $p$ is in the closure of $\mathfrak{B}(\rho n_1)  \cap  \mathfrak{B}(\rho n_2)$, for $\epsilon > 0$,
		\begin{equation*}
		\ball{\epsilon\cos(\theta(\epsilon))}{p} \cap  \mathfrak{B}(\rho n_1)  \cap \mathfrak{B}(\rho n_2)  \neq \emptyset.
		\end{equation*}
		Finally, since this subset is a non-empty intersection of open sets, it is also open, and it has positive volume.
		
	\end{proof}
	
	\begin{proposition} \label{prop:regintproj}
		Consider $E = \sublevel{f_1}{0} \cap \sublevel{f_2}{0}$, and suppose $f_1$ and $f_2$ satisfy the conditions placed on $f$ in \Cref{lem:sublevelproj}, and $\rho$ be as defined in \Cref{lem:sublevelproj}. Then, if $\epsilon < \rho$, and $p \in  E$, then
		\begin{equation*}
		\vol\qty(\ball{\epsilon}{p} \cap E) \geq  \vhstwo{\frac{\overline{\eta}_E(\epsilon)}{2} \cos(\theta\qty(\frac{\overline{\eta}_E(\epsilon)}{2})))}{\rho}{\phi_{12}} > 0
		\end{equation*}
		where $\phi_{12}$, $\vhstwo{\cdot}{\cdot}{\cdot}$, and $\overline{\eta}_E$ are as defined in \Cref{lem:regintproj2}, \Cref{nota:balls}, and \Cref{prop:bottlethickbound} respectively.
	\end{proposition}
	
	\begin{proof} Combining \Cref{lem:regintproj1} and \Cref{lem:regintproj2}, we have, for $p \in E$ and $\epsilon < \rho$,
		\begin{equation*}
		\vol\qty(\ball{\epsilon}{p} \cap E) \geq
		\begin{cases}
		\vhstwo{\epsilon \cos(\epsilon)}{\rho}{\phi_{12}} & \qq{if} p \notin E_\epsilon \\
		& \\
		\vhs{\frac{\overline{\eta}_E(\epsilon)}{2}\cos(\theta\qty(\frac{\overline{\eta}_E(\epsilon)}{2})))}{\rho}  & \qq{otherwise.}
		\end{cases}
		\end{equation*}
		Since $\epsilon > \overline{\eta}_E(\epsilon)$ (\Cref{prop:bottlethickbound}), and for any $t > 0$,
		\begin{align*}
		\vhstwo{t}{\rho}{\phi_{12}}  &= \vol\qty(\mathscr{B}_{t\cos(\theta(t))} \cap \mathfrak{B}(\rho n_1) \cap \mathfrak{B}(\rho n_2)) \\
		&\leq \vol\qty(\mathscr{B}_{t\cos(\theta(t))} \cap \mathfrak{B}(\rho n_1)) \\
		&= \vhs{t \cos(\theta(t))}{\rho},
		\end{align*}
		our bound holds for either $p \in E_\epsilon$ or otherwise.
	\end{proof}

	\section{Technical Lemmas}

	\begin{lemma}\label{lem:measurecont}
		Let $A$ be measurable subset of a Riemannian manifold $M \subset \R^d$ with $\mu_g(A) > 0$, where $\mu_g$ is the Riemannian density on $M$. If $\nu$ is a measure supported on $M$ that is continuous with respect to $\mu_g$, then for any radius $r > 0$ the function $p \mapsto \nu(\ball{r}{p}) \cap A$ is continuous on $A$.
		\begin{proof}
			Consider $c$ such that $\norm{c-p}= \delta < r$. Then
			\begin{align*}
			\ball{r-\delta}{p} &\subset \ball{r}{c} \subset \ball{r+\delta}{p} \\
			\implies \nu(\ball{r-\delta}{p} \cap A) &\leq  \nu(\ball{r}{c} \cap A) \leq \nu(\ball{r+\delta}{p} \cap A).
			\end{align*}
			As $\nu(\ball{r+\delta}{p} \cap A)$ monotonically increases with $\delta$, as we decrease $\delta$, the monotonicity and continuity of the measure with respect to $\delta$ ensures that  $\nu(\ball{r\pm\delta}{p} \cap A) \xrightarrow{\delta \to 0} \nu(\ball{r}{p} \cap A)$. Thus, by the sandwich theorem,
			\begin{equation*}
			\lim_{c \to p} \nu(\ball{r}{c} \cap A) = \nu(\ball{r}{p} \cap A).
			\end{equation*}
			Thus, $\nu(\ball{r}{p} \cap A)$ is continuous with respect to $p$.
		\end{proof}
	\end{lemma}
	
	For $M \subset \R^d$ Recall that $\zeta_p:M \to T_pM$ is the orthogonal projection onto the $m$-dimensional plane tangent to $M$ at $p$. 
	\begin{restatable}{lemma}{projdet} \label{lem:projdet}%
		Let $\dd_q{\zeta_p}$ be the Jacobian of $\zeta_p$ at $q \in M$ with respect to orthonormal coordinates in $T_qM$ and $\tp{p}{M}$. Then $\abs{\det \dd{_q\zeta_p}} \leq 1$.
	\end{restatable}
	
	\begin{proof}
		Let $P: \R^d \to \R^m$ be the orthogonal projection onto the $m$-dimensional plane in $\R^d$ spanning the first $m$ coordinates. As $F$ is the restriction of $P$ to $M$, therefore the derivative $\dd{F}: \Tb{M} \to \R^m$ between tangent bundles is the restriction of $\dd{P}: \Tb{\R^d} \to \R^m$ to $\Tb{M}$. We have $\dd{P} = 1_m \oplus 0_{d-m}$ where $1_m$ is the identity matrix corresponding to the first $m$ coordinates, and $0_{m,d-m}$ is the $m\times d-m$ matrix of zeros. Since $\dd{_pF}$ is a restriction of $\dd{_pP}$ to an $m$-dimensional subspace $T_pM \subset T_p \R^d$, the absolute value of the determinant of $\dd{_pF}$ is at most 1.
	\end{proof}

	\bibliographystyle{abbrv}
	\bibliography{refs}
	
\end{document}